\documentclass[nohyperref, referee, sn-mathphys,Numbered]{sn-jnl}

\usepackage{graphicx}%
\usepackage{multirow}%
\usepackage{amsmath,amssymb,amsfonts}%
\usepackage{amsthm}%
\usepackage{mathrsfs}%
\usepackage{xcolor}%
\usepackage{textcomp}%
\usepackage{manyfoot}%
\usepackage{enumitem, microtype, lineno, orcidlink, mleftright}
\usepackage{listings}

\newtheorem{theorem}{Theorem}
\newtheorem{lemma}{Lemma}
\newtheorem{coro}{Corollary}

\usepackage{tikz}
\usepackage{adjustbox} 

\usetikzlibrary{shapes.geometric, arrows}

\tikzstyle{startstop} = [rectangle, rounded corners, minimum width=3cm, minimum height=1cm, text centered, draw=black, fill=gray!10] 
\tikzstyle{process} = [rectangle, minimum width=3cm, minimum height=1cm, text centered, draw=black, fill=gray!20] 
\tikzstyle{decision} = [diamond, minimum width=3cm, minimum height=1cm, text centered, draw=black, fill=gray!30] 
\tikzstyle{arrow} = [thick,->,>=stealth]

\let\svthefootnote\thefootnote
\newcommand\freefootnote[1]{%
  \let\thefootnote\relax%
  \footnotetext{#1}%
  \let\thefootnote\svthefootnote%
}

\begin{document}

\title{A Novel Generalization of the Liouville  Function $\lambda(n)$ and a Convergence Result for the Associated Dirichlet Series}


\author{\fnm{Sky Pelletier} \sur{Waterpeace} \orcidlink{0000-0002-2231-0160}}

\affil{\orgname{Rowan University}, \orgaddress{\city{Glassboro}, \state{NJ}, \postcode{08028}, \country{USA}}}
\affil{\orgname{Southern New Hampshire University}, \orgaddress{\city{Manchester}, \state{NH}, \postcode{03106}, \country{USA}}\freefootnote{\textbf{Statements and Declarations} This research did not receive any specific grant from funding agencies in the public, commercial, or not-for-profit sectors and there are no interests to declare. Affiliation information is for contact purposes only and does not imply any support or affiliation with this work.}}

\email{waterpeace@rowan.edu}\email{s.waterpeace@snhu.edu}


\abstract{We introduce a novel arithmetic function $w(n)$, a generalization of the Liouville function $\lambda(n)$, as the coefficients of a Dirichlet series.  By spatially encoding information in a natural way about the distribution of prime factors among natural numbers, $w(n)$ allows results to be obtained which rely intrinsically on the distribution of primes without having direct knowledge of that distribution.  We prove some properties of the distribution of $w(n)$ and then provide a result on the convergence of its Dirichlet series.
A parametrized family of functions $w_m(n)$ is defined of which $w(n)$ is a special case.   We show that each function $w_m(n)$ injectively maps $\mathbb{N}$ into a dense subset of the unit circle in $\mathbb{C}$ and that each $F_m(s) = \sum_n \frac{w_m(n)}{n^s}$ converges for all $s$ with $\Re(s)\in\left(\frac{1}{2},1\right)$.  Finally, we show that the family of functions $w_m(n)$ converges to $\lambda(n)$ and that $F_m(s)$ converges uniformly in $m$ to $\sum_n \frac{\lambda(n)}{n^s}$, implying convergence of that series in the same region and thereby proving an interesting property about a closely related function.}


\maketitle

\newpage

\section{Introduction}

In this paper, a novel arithmetic function $w(n)$ is introduced which spatially encodes information about the distribution of prime factors among natural numbers in a natural way and which allows results to be obtained which rely intrinsically on the distribution of primes without having direct knowledge of that distribution.

\paragraph{Definition of $F(s)$ and $w(n)$} 

We define an arithmetic function $w(n)$ as the coefficients of the formal Dirichlet series given by
\begin{linenomath*}\begin{equation}\label{F1}
F(s) = \sum_{n=1}^\infty \frac{w(n)}{n^s} := 
    \prod_{p\in\mathbb{P}} \left( 1 - \frac{e^{i\psi(p)}}{p^s} 
    +\frac{e^{i\psi(p^2)}}{p^{2s}}
    - \frac{e^{i\psi(p^3)}}{p^{3s}}
    + \cdots \right),
\end{equation}\end{linenomath*}
where $\mathbb{P} = \{p_j\}_{j\in\mathbb{N}}$ is the set of ordered primes and for each $p\in\mathbb{P}$, $\psi(p^k)$ is given by 
\begin{linenomath*}\begin{equation}\label{psi_def}
\psi(p^k) := \frac{\pi}{p^2G}\left(1 - \left(\frac{p-1}{p}\right)^k\right),
\end{equation}\end{linenomath*}
in which we are letting $G=\sum_{p\in\mathbb{P}} p^{-2}$ be the sum of the reciprocals of the primes squared.\footnote{That is, we use $G$ as shorthand for $P(2)$, where $P$ is the prime zeta function.}    We observe that $F(s)$ is a kind of Euler product defining each $w(n)$ based on the prime factorization of $n$ and that (\ref{psi_def}) naturally gives $\psi(p^0)=0$ and hence $w(1)=1$.  Although we are considering the series defining $F(s)$ formally, we note that each factor on the right of (\ref{F1}) can easily be seen to converge geometrically for any $p\in\mathbb{P}$ and $s>0$, since we have
\begin{linenomath*}
    \begin{equation}
        \left|\sum_{k=0}^{\infty} (-1)^k \frac{ie^{\psi(p^k)}}{p^{ks}}\right|
        \leq \sum_{k=0}^{\infty} \frac{1}{p^{ks}} = \frac{1}{1-p^{-s}}.
    \end{equation}
\end{linenomath*}
On the other hand, the convergence of the left-hand side of (\ref{F1}) in the same region is a much more interesting question and is a main result of this paper.

\begin{lemma}\label{properties}
$w(n)$ has the following properties:
\begin{enumerate}
    \item \label{multiplicative} $w(n)$ is multiplicative: $w(ab)$ = $w(a)w(b)$ for any coprime $a,b\in\mathbb{N}$
    \item \label{injective} $w:\mathbb{N}\to C$ is injective, where $C = \{z\in\mathbb{C}\ \big|\ |z|=1\}$, the unit circle in $\mathbb{C}$.
    \item \label{llike} $\arg(w(n))\in[0,\pi)$ for $\lambda(n)=1$ and $\arg(w(n)) \in(\pi,2\pi)$ for $\lambda(n)=-1$, where $\lambda(n)$ is the Liouville function.
\end{enumerate}
\end{lemma}

\begin{proof}
Properties \ref{multiplicative} and \ref{injective} are obvious.  We therefore proceed by proving Property \ref{llike}. 

Observe that if the prime factorization of $n>1$ is given by $n=p_1^{k_1} p_2^{k_2} \cdots p_J^{k_J}$, then we have
\begin{linenomath*}\begin{align}
    \frac{w(n)}{n^s} & = \frac{\mleft(-1\mright)^{k_1}e^{i\psi\mleft(p_1^{k_1}\mright)}}{p_1^{k_1 s}} 
        \times    
    \frac{\mleft(-1\mright)^{k_2}e^{i\psi\mleft(p_2^{k_2}\mright)}}{p_2^{k_2 s}}
        \times\cdots\times
    \frac{\mleft(-1\mright)^{k_J}e^{i\psi\mleft(p_J^{k_J}\mright)}}{p_J^{k_J s}} \\
    & = \frac{\mleft(-1\mright)^{k_1 + k_2+ \cdots + k_J} e^{i\mleft(\psi\mleft(p_1^{k_1}\mright)+
    \psi\mleft(p_2^{k_2}\mright)+\cdots+
    \psi\mleft(p_J^{k_J}\mright)\mright)}}
    {p_1^{k_1s}p_2^{k_2s}\cdots p_J^{k_Js}} \\
    &= \frac{\lambda(n)
        e^{i\mleft(\psi\mleft(p_1^{k_1}\mright)+
    \psi\mleft(p_2^{k_2}\mright)+\cdots+
    \psi\mleft(p_J^{k_J}\mright)\mright)}}
    {\mleft(p_1^{k_1}p_2^{k_2}\cdots p_J^{k_J}\mright)^s} \\
    & = \frac{e^{i\frac{\pi}{2}(1-\lambda(n))}
        e^{i\mleft(\psi\mleft(p_1^{k_1}\mright)+
            \psi\mleft(p_2^{k_2}\mright)+\cdots+
            \psi\mleft(p_J^{k_J}\mright)\mright)}}
        {n^s},
\end{align}\end{linenomath*}
where $\lambda(n)$ is the Liouville lambda function.  We thus have 
$w(n)=e^{i\theta(n)}$, in which
\begin{linenomath*}\begin{equation}\label{argument}
    \theta(n) = \arg(w(n))=\frac{\pi}{2}\left(1-\lambda(n)\right) + \sum_{j=1}^J \psi\bigl(p_j^{k_j}\bigr).
\end{equation}\end{linenomath*}

It is sufficient to observe that
\begin{align}
    \sum_{j=1}^J \psi\bigl(p_j^{k_j}\bigr) & = 
        \sum_{j=1}^J \frac{\pi}{p_j^2 G} \left(1 - \left(\frac{p_j -1}{p_j}\right)^{k_j}\right)\\
    & = \pi \frac{ \frac{1}{p_1^2} \left(1 - \left(\frac{p_1 -1}{p_1}\right)^{k_1}\right) + \cdots + \frac{1}{p_J^2} \left(1 - \left(\frac{p_J -1}{p_J}\right)^{k_J}\right)}
    {\frac{1}{2^2} +\frac{1}{3^2} +\frac{1}{5^2} + \cdots + \frac{1}{p_J^2} + \cdots}.\label{numerator}
\end{align}
Since $k_j \in \mathbb{N}$, we have that $0 < \left(1 - \left(\frac{p_j -1}{p_j}\right)^{k_j}\right) < 1$ for all $1 \leq j \leq J$.  It follows that the numerator in (\ref{numerator}) is always strictly smaller than the denominator.  Thus, $0 < \sum_{j=1}^J \psi\bigl(p_j^{k_j}\bigr) < \pi$, and the cases for $[0,\pi)$ and $(\pi,2\pi)$ both follow from this observation, taking the sum of the $k_j$'s even or odd, respectively, and recalling that $w(1)=1$.  This proves Property \ref{llike} and completes the proof of Lemma \ref{properties}.
\end{proof}


\section{Distribution of $w(n)$}\label{distribution}


\paragraph{Remark}

The structure of $w(n)$ as defined in (\ref{F1}) is quite rich.  This is not surprising, since  $w(n)$ is specifically designed to spatially encode information about the distribution of natural numbers based on their prime factorizations.
Though the exact distribution of points $w(n)$ in the unit circle in $\mathbb{C}$ is very intricate and is in many ways difficult to pin down, we will in this section provide several results.   It is (relatively) easy to observe that $w(n)$ has fractal structure, with the image of any subset $A$ of $\mathbb{N}$ infinitely repeated in scaled pseudocopies by multiplying every $a\in A$ by other natural numbers, for example by powers of $2$. 

Other observations require considerably more care.  In this section, we will prove that the set $\{w(n)\}_{n\in\mathbb{N}}$ is dense in the unit circle in $\mathbb{C}$, and the proof of this theorem will begin to reveal some of the subtle structure of $w(n)$.  A technical refinement of the density result, Hughes' Corollary, will be important for the convergence results in the next section, as will a result on the natural densities of elements $n\in\mathbb{N}$ whose images are in opposite arcs of the unit circle.  

\begin{theorem}[Density Theorem]\label{density}
The set of coefficients $\{w(n)\}$ for $n\in\mathbb{N}$ is dense in $C$, the unit circle in $\mathbb{C}$.
\begin{proof}[Sketch of proof]
It is straightforward to see how density can potentially be proven.  However, due to the necessity of dealing with primes whose values are unknown, certain details must be handled carefully.  The full proof is lengthy and fairly technical and is provided in Appendix A, and SageMath code demonstrating the algorithm in the full proof is provided in Appendix B.  We present here a sketch of a simple proof and refer the reader to the appendices for full details.\footnote{Special thanks to Kevin Hughes for this clarifying sketch.}

We will prove the equivalent statement that $\{\theta(n)\}_{n\in\mathbb{N}}$ is dense in $[0,2\pi)$.  Recall that we are letting $\mathbb{P} = \{p_j\}_{j\in\mathbb{N}}$ be the set of ordered primes, so that $p_1 =2$, $p_2=3$, $p_3=5$, \ldots, and that we define
\begin{equation}
    G=\sum_{p\in\mathbb{P}} \frac{1}{p^{2}} = 0.4522474200\cdots
\end{equation}
Without loss of generality, let $x\in[0,\pi)$.  If $x$ is the argument of $w(n)$ for some $n$, then $x$ must be of the form of the right hand side of (\ref{numerator}).  Hence, the general goal is to find primes $p_j$ with exponents $k_j$ so that the approximation
\begin{linenomath*}
    \begin{equation}\label{dens_approx}
        x \approx \frac{\pi}{G} \left[
        \frac{1}{p_1^2}\left(1-\left(\frac{p_1-1}{p_1}\right)^{k_1}\right) + \cdots +
        \frac{1}{p_J^2}\left(1-\left(\frac{p_J-1}{p_J}\right)^{k_J}\right) \right]
    \end{equation}
\end{linenomath*}
can be made arbitrarily precise for any given $x\in[0,\pi)$.  For $x$ close to $\pi$, this is trivial.  The difficulty comes in the fact that we do not know the exact contribution of nonzero exponents $k_j$ for any but a finite number of values $j$ corresponding to known primes.  

For a given $p\in\mathbb{P}$, let $\hat{\psi}(p)$ be its limiting maximum contribution to the right side of (\ref{dens_approx}), so that $\hat{\psi}(p) = \frac{\pi}{p^2 G}$.  Depending on the choice of $k$, $p^k$ contributes to $\theta(n)$ from zero up to something arbitrarily close to $\hat{\psi}(p)$.

The idea is to create a collection $I$ of indices of primes $p_j$ so that
\begin{linenomath*}
    \begin{equation}\label{sketch_sum}
        S = \frac{\pi}{G} \sum_{j\in I} \frac{1}{p_j^2}
    \end{equation}
\end{linenomath*}
is within an $\epsilon$-radius of $x$ for any chosen $\epsilon>0$.   By then choosing each $k_j$ sufficiently large, we can find arbitrarily many numbers $n$ for which the arguments of $w(n)$ are in the desired interval, which we refer to  as the \emph{target interval} or simply the \emph{target}.

Let $I=\emptyset$.  Start with $p_1$ and examine each prime in order, keeping track of the running sum $S$ from (\ref{sketch_sum}) as well as the sum of the $\hat{\psi}(p)$ for all the primes not yet examined, which we shall refer to as the \emph{tail} to draw from in the future. The current $p_j$ under consideration is not included in the tail.

Consider $S+\hat{\psi}(p_j)$.  If this sum is still less than the target interval, include $j$ in $I$ so that $\hat{\psi}(p_j)$ is now included in $S$.  On the other hand, if $S+ \hat{\psi}(p_j)$ is greater than the target interval, choose the largest $k_j$ so that the new running total is less than or equal to the target and instead add to $S$ only $\psi(p_j^{k_j})$ (with appropriate bookkeeping to keep track of which $j\in I$ do not have all of $\hat{\psi}(p_j)$ added).

What would be required for this algorithm to fail? For each $p_j$, prior to selecting $k_j$, as long as the running total plus $\hat{\psi}(p_j)$ plus the tail is greater than the target, the algorithm hasn't failed yet. Note that $\hat{\psi}(p_j)$ was part of the tail during the selection for the previous $p_{j-1}$.

Critically, failure at the $p_{j+1}$ step would require both (a) $S +\psi(p_j^{k_j}) + tail < target$, and (b) $S+ \psi(p^{(k+1)}) > target$.
However, by Bertrand's postulate we have that $\hat{\psi}(p_{j+1}) > \frac{1}{4}\hat{\psi}(p_j)$.  It is therefore clear that the size of the tail exceeds  $\frac{1}{4}\hat{\psi}(p_j)$. The question then is whether our selection of $k$ can lower the distance to the target interval to less than $\frac{1}{4}$ of $\hat{\psi}(p_j)$; it can, since incrementing $k$ from $0$ to $1$ only adds $\frac{1}{p}\hat{\psi}(p_j)$, which is less than $\frac{1}{4}\hat{\psi}(p_j)$ unless $p_j$ is 2 or 3, and the subsequent $k$'s contribute less than the first, for each $p$.  (For 2 and 3 we can numerically verify that the tail is big enough to avoid failure.)

Since the algorithm therefore cannot terminate in failure, it succeeds in approaching a limit in the target interval.  Since we only need to get inside the interval, there is a finite collection of $p_j$ with exponents $k_j$ which will be close enough.\footnote{Note again that in Appendix A a constructive proof is presented with full details.  The algorithm in the full proof considers batches of the terms of $G$ for consecutive primes simultaneously and only considers specific values of $k$ as needed until the end, since this turns out to be a special case.  SageMath code demonstrating the algorithm of the full proof is presented in Appendix B.}
\end{proof}
\end{theorem}


We now introduce two important technical consequences of the density of the set $\{w(n)\}$ which will be necessary for the convergence proofs in the next section.  

\begin{coro}[Hughes' Corollary\footnote{The author wishes to express his immense gratitude to his dear friend, Kevin Hughes, whose help with the proof of this assertion was indispensable, and whose decades of friendship are a tremendous blessing.}]\label{subset_dense}
    For every $N\in\mathbb{N}$, define $\Theta_N$ to be the set of principal arguments of $w(n)$ for all $n\leq N$.  Then for any $\epsilon > 0$, there exists $N_M$ sufficiently large that for all $N>N_M$, any interval $(a, a+\epsilon) \subset [0,\pi) \cup (\pi,2\pi)$ contains at least $N^x$ elements of $\Theta_N$ for any $x\in[0,1)$.
\begin{proof}
    For any $N\in\mathbb{N}$, define
    \begin{eqnarray}\label{bigtheta}
        \Theta_N = \left\{\theta(n)\ \big|\ n\leq N\right\},
    \end{eqnarray}
    where each $\theta(n)$ as defined in (\ref{argument}) is the principal argument of $w(n)$.
    Let $\epsilon>0$.  Without loss of generality, let $I$ be any $\epsilon$-width interval ${(a,a+\epsilon) \subset [0,\pi)}$. 
    The Density Theorem implies there exists $N_0 \in \mathbb{N}$ large enough that for some $n<N_0$, $\theta(n) \in (a+\pi,\ a+\pi+\frac{\epsilon}{2})$.

    Let $p_j\in\mathbb{P}$ be smallest such that $np_j>N_0$ and $\theta(p_j)<\frac{\epsilon}{2}$.  Note that $\psi(p_k) < \psi(p_j)$ for any $k>j$, and so it follows from Lemma \ref{properties} that $\theta(np_k)\in I$ for all $k\geq j$.  Hence, for any $M\in\mathbb{N}$, 
    \begin{equation}
        \left\{\theta\left(np_{j+m}\right) \big|\ m\in \mathbb{N}, m\leq M\right\} \subset I,
    \end{equation}
    and so when $N\geq p_{j+M}N_0$, there are at least $M$ elements of $\Theta_N$ in $I$.  For notational simplicity, let $P_M$ be the $(j+M)$th prime, so $P_M=p_{j+M}$, and let $N_M=P_M N_0$.  Note that by the Prime Number Theorem,
    \begin{equation}
        j+M \thicksim \frac{P_M}{\log P_M},
    \end{equation}
    and so
    \begin{equation}
        M \thicksim \frac{P_M}{\log P_M} - j.
    \end{equation}
    To show $M>N_M^x$ for all sufficiently large $M$, consider the ratio $\frac{M}{N_M^x}$.  We have that
    \begin{align}
        \frac{M}{N_M^x} &\thicksim \frac{ \frac{P_M}{\log P_M} - j}{N_M^x}
         = \frac{ \frac{P_M}{\log P_M} - j}{P_M^x N_0^x}\\
        &= \left(\frac{1}{N_0^x}\right) 
            \frac{\frac{P_M}{\log P_M} - j}{P_M^x}\\
        &= \left(\frac{1}{N_0^x}\right)\left(\frac{P_M^{1-x}}{\log P_M} - \frac{j}{P_M^x}\right).
    \end{align}
    Since $j,N_0$ are fixed for any given $\epsilon$, it is clear that we can choose $M$ large enough so that this ratio exceeds any positive number.  

    Our choice of $I$ to be contained in $[0,\pi)$ was arbitrary, and the case where $I$ is instead contained in $(\pi, 2\pi)$ is similar.  It follows that for any $\epsilon$-width interval contained in $[0, \pi) \cup (\pi, 2\pi)$, there exists $N_M$ sufficiently large that for all $N>N_M$ and any $x\in[0,1)$, there are at least $N^x$ elements $\theta(n)$, $n \leq N$, in that interval.
\end{proof}
\end{coro}


\begin{coro}\label{nat_dense}
    For any $x,y\in(0,\pi), x<y$, let
    \begin{eqnarray}
        A &=&\left\{ n \in \mathbb{N}\ \big|\ \theta(n) \in \left(x,y\right)\right\}\\
        B &=& \left\{ n \in \mathbb{N}\ \big|\ \theta(n) \in \left(x+\pi,y+\pi\right)\right\},
    \end{eqnarray}
    and define the natural density of a set $S$ in $\mathbb{N}$ by
    \begin{equation}
        \delta(S)=\lim_{N\to\infty} \frac{\left|\left\{n<N\ \big|\ n \in S\right\}\right|}{N},
    \end{equation}
    if that limit exists.
    Then $\delta(A)$ and $\delta(B)$ exist, are nonzero, and are equal.
\begin{proof}
    Let $x\in(0,\pi)$, and define
    \begin{eqnarray}
        A_x &=& \left\{n\in\mathbb{N}\ \big|\ \theta(n)\in \left(x,\pi\right)\right\}\\
        B_x &=& \left\{n\in\mathbb{N}\ \big|\ \theta(n)\in \left(\pi+x,2\pi\right)\right\}.
    \end{eqnarray}
    Now let $X = A_x \cup B_x$.  $X$ is a subset of $\mathbb{N}$, and since the Density Theorem implies that $X$ is nonempty, it has a least element; call it $d_0$, and let $X_0 = d_0\mathbb{N} = \left\{d_0 n\ \big|\ n\in\mathbb{N}\right\}$.  Note that $\delta(X_0)=\frac{1}{d_0}$.  The Prime Number Theorem implies that the natural density of $n\in\mathbb{N}$ for which $\lambda(n) =1$ equals the natural density of $m\in\mathbb{N}$ for which $\lambda(m) = -1$, and it follows that $\delta(X_0 \cap A_x) = \delta(X_0 \cap B_x)$.

    We have, again from the Density Theorem, that $X - X_0$ is not empty, so let $d_1$ be its smallest element and define $X_1 = d_1\mathbb{N}$ similarly to above, and note that PNT again implies $\delta(X_1 \cap A_x) = \delta(X_1 \cap B_x)$.  Let $m_1=\mathrm{lcm}(d_0,d_1)$ be the least common multiple, and note that if we similarly define $M_1 = m_1\mathbb{N}$ then we have that $M_1 = X_0 \cap X_1$ and that $\delta(M_1\cap A_x) = \delta(M_1 \cap B_x)$.  Since natural density of disjoint sets is additive, we have that
    \begin{eqnarray}
        \delta(X_0 \cap A_x) &=& \delta(X_0 \cap A_x \cap M_1) + \delta(X_0 \cap A_x \cap M_1^c)\\
        \delta(X_1 \cap A_x) &=& \delta(X_1 \cap A_x \cap M_1) + \delta(X_1 \cap A_x \cap M_1^c)
    \end{eqnarray}
    and likewise if we replace $A_x$ with $B_x$.   It follows that
    \begin{equation}
        \delta\left(\left(X_0 \cup X_1\right) \cap A_x\right)= \delta\left(\left(X_0 \cup X_1\right) \cap B_x\right).
    \end{equation}
    Continue in this way to let $d_2$ be the least element of $X - (X_0 \cup X_1)$, letting $X_2 = d_2\mathbb{N}$ and $m_2 = \mathrm{lcm}(m_1, d_2)$, concluding ultimately that for all $k\in\mathbb{N}$
    \begin{equation}
        \delta\left(\bigcup_k X_k \cap A_x\right) = \delta\left(\bigcup_k X_k \cap B_x\right),
    \end{equation}
    implying that $\delta(A_x) = \delta(B_x)$.  
    Now let $y\in(0,\pi)$ with $x<y$ and define $A_y,B_y,Y$ similarly to above.  Then if we define $W = X - Y$, $A=A_x - A_y$, and $B = B_x - B_y$,  we have that $\delta\left(W \cap A\right) = \delta\left(W \cap B\right)$ and so $\delta(A) = \delta(B)$ as desired.
\end{proof}
\end{coro}



\section{Convergence}

\paragraph{Remark}
$F(s)$ was introduced as a formal series, but in this section we obtain specific convergence results.  We begin by identifying values $s\in\mathbb{C}$ for which $F(s)$ is pointwise convergent.  We then generalize $F(s)$ and $w(n)$ to a family of related functions and extend the convergence results to the family $F_m(s)$, $m\in\mathbb{N}$.  Finally, after some technical preliminaries, we show that the family $F_m(s)$ converges to a well-known function, demonstrating convergence results for that function.

\begin{theorem}[Convergence of $F_N(s)$ for $\Re(s)\in\left(\frac{1}{2},1\right)$]\label{convthm}
Define
\begin{linenomath*}\begin{equation}
F_N(s) := \sum_{n=1}^{N} \frac{w(n)}{n^s},
\end{equation}\end{linenomath*} where $w(n)$ is as defined in (\ref{F1}).  Then
\begin{linenomath*}\begin{equation}\label{conv}
\lim_{N\to\infty} F_N(s) = F(s) = \sum_{n=1}^{\infty} \frac{w(n)}{n^s}
\end{equation}\end{linenomath*}
converges pointwise for all $s$ with $\Re(s)\in\left(\frac{1}{2},1\right)$.
\end{theorem}
\begin{proof}
    We will take advantage of the following well-known fact:
\begin{quote}
\textbf{Proposition 1.7.7, p43 in \cite{o._2007}}
Let $a(n)$ be a sequence, and define $A(x) = \sum_{n\leq x} a(n)$.  If $|A(x)| \leq Mx^\alpha$ for all $x \geq 1$, where $\alpha \geq 0$, then [the Dirichlet series] $\sum_{n=1}^\infty \frac{a(n)}{n^s}$ is convergent for all $s=\sigma + it$ with $\sigma>\alpha$.\\
\end{quote}

We will show that there exists an $N_0$ and $M$ such that for all $N>N_0$, $\big|\sum_{n=1}^N w(n)\big| < MN^\alpha$ for $\alpha \in \left(\frac{1}{2},1\right)$.%
    \footnote{Restricting $\alpha$ to $\left(\frac{1}{2},1\right)$ may at this point appear somewhat arbitrary, but it is directly related to the limitation in our claim, which  asserts only the convergence of $F(s)$ for $\Re(s)\in\left(\frac{1}{2},1\right)$.  The wisdom of this restriction will hopefully be made clear to the reader by the end of the paper.}  
With the proposition above, this will demonstrate the convergence desired.  

Fix $\alpha \in \left(\frac{1}{2},1\right)$, and for each $N \in \mathbb{N}$ define $J=\lfloor N^\alpha \rfloor$ and $K=\lfloor N / \lfloor N^{\alpha}\rfloor\rfloor$.  Then $N=JK+R_N$, with $0 \leq R_N<J$.  With these assumptions, we have 
\begin{equation}
    0< 1-\alpha < \frac{1}{2} < \alpha < 1,
\end{equation}and hence
\begin{equation}\label{implications_of_defs_JK}
    1 \leq \underbrace{\lfloor N^{1-\alpha}\rfloor}_{K} < N^{1-\alpha} < N^{\frac{1}{2}}
      \leq \underbrace{\lfloor N^{\alpha}\rfloor}_{J} < N,
\end{equation}
noting however that the two $\leq$'s may achieve equality only possibly for small $N$.\footnote{The reader may wonder about the justification for these definitions of $J$ and $K$.  In Section \ref{distribution}, we explored the structure of $w(n)$ and, in particular, the distribution of arguments of $w(n)$ in $[0,2\pi)$.  Ultimately, the purpose of these definitions for $J,K$ will be to demonstrate that we can capture significant information about that distribution by examining only $J$ appropriately chosen points out of $N$.}

Let
\begin{equation}
    \Theta_N = \{\arg(w(n)) \in [0,2\pi), n\leq N\}
\end{equation}be the set of principal arguments of $w(n)$ with $n\leq N$, as in the proof of Corollary \ref{subset_dense}, and let $\Theta_N^* = \Theta_N - \Theta_{R_N}$ be the set after removing the arguments of $w(n)$ for $n\leq R_N$.  Order the $JK$ elements of $\Theta_N^*$ such that 
\begin{linenomath*}\begin{equation}\label{orderthetas}\begin{split}
    \Theta_N^*:\hfill 
        0\leq \overbrace{\theta_{1,1} < \theta_{1,2} < \cdots < \theta_{1,K}}^{j=1} < 
            &\cdots < \overbrace{\theta_{j,1} < \theta_{j,2} < \cdots < \theta_{j,K}}^j < \cdots \\
            &\cdots < \overbrace{\theta_{J,1} < \theta_{J,2} < \cdots < \theta_{J,K}}^{j=J} < 2\pi. 
\end{split}\end{equation}\end{linenomath*}
Letting $N^* = \left|\Theta_N^*\right|= JK$, we have 
\begin{linenomath*}\begin{equation}
\sum_{n=R_N+1}^{R_N + N^*} w(n) = \sum^J\sum^K e^{i\theta_{j,k}} = \sum^K\sum^J e^{i\theta_{j,k}}.
\end{equation}\end{linenomath*}
Therefore,
\begin{linenomath*}\begin{align}
    \Bigg|\sum_{n=R_N+1}^{R_N + N^*} w(n)\Bigg| & \leq 
        \Bigg|\sum^J e^{i\theta_{j,1}}\Bigg| + \cdots + \Bigg|\sum^J e^{i\theta_{j,K}}\Bigg|\label{sep_sums} \\
    &  \leq K \cdot \max_{1\leq k \leq K} \Bigg|\sum^J e^{i\theta_{j,k}}\Bigg|,
\end{align}\end{linenomath*}
where by $\max_{1\leq k \leq K}$ we are selecting $k$ to be the value which yields the maximum among the sums on the right of (\ref{sep_sums}). 
For such $k$, redefine $\theta_j := \theta_{j,k}$ so that 
\begin{linenomath*}
    \begin{equation}\label{choose_one_k}
        \Big|\sum^J e^{i\theta_{j}}\Big| = \max_{1\leq k \leq K} \Big|\sum^J e^{i\theta_{j,k}}\Big|,
    \end{equation}
\end{linenomath*} and we have 
\begin{linenomath*}\begin{equation}\label{newbound}
\Bigg|\sum_{n=R_N+1}^{R_N + N^*} w(n)\Bigg| \leq K\Bigg|\sum^J e^{i\theta_{j}}\Bigg| = \Bigg|\sum^J e^{i\theta_{j}}K\Bigg|.
\end{equation}\end{linenomath*}
Note that the selection of $k$ which maximizes the sum on the right of (\ref{sep_sums}) is dependent on $N$, and so for each $N$, we have a well-defined\footnote{%
    It is of course possible, though unlikely for large $N$, that multiple values $k$ may yield an equivalent maximum, in which case any may be chosen without affecting the argument.}
collection ${C_N = \{\theta_j\ \big|\ j=1,\cdots,J\}_N \subset \Theta_N^*}$, the elements of which we expect to vary, perhaps substantially, for various $N$.  However, for any $N$, (\ref{newbound}) holds, and it follows that 
\begin{linenomath*}\begin{equation}\label{mult_by_2pi}
\frac{2\pi}{J}\Bigg|\sum_{n=R_N+1}^{R_N + N^*} w(n)\Bigg| \leq \frac{2\pi}{J}\Bigg|\sum^J e^{i\theta_{j}}K\Bigg| = \Bigg|\sum^J e^{i\theta_{j}}\frac{K2\pi}{J}\Bigg|.
\end{equation}\end{linenomath*}

Consider the partition of $[0, 2\pi)$ induced by $C_N$.  Note that while well-defined, the specific values of the elements of $C_N$ are unknown and depend on the prime factorizations of the $J$ numbers $n$ (with $R_N < n \leq N$) underlying the arguments that form $C_N$.  Nevertheless, Hughes' Corollary guarantees that the norm of this partition goes to zero as $N$ goes to infinity.  That is, for sufficiently large $N$, the elements $\theta_j \in C_N$ may ``bunch up'' but there will not be any gaps.

We will now extend the partition $C_N$ to cover the interval $[0,K2\pi)$, essentially by adding multiples of $2\pi$ to the elements and reordering and reindexing them, which changes the expressions of their arguments but not the locations in the unit circle of each corresponding $w(n)$.  First, replace each of the $\theta_j$ with $\{\theta_j, \theta_j+2\pi, \theta_j+4\pi, \cdots, \theta_j + (K-1)2\pi\}$.  We now have a collection
\begin{linenomath*}
\begin{equation}\begin{split}\label{spread_to_K_ints}
    0 \leq \theta_1 < \cdots < \theta_J &< \theta_1 +2\pi < \cdots < \theta_J + 2\pi  < \cdots \\
   &< \theta_1 + (K-1)2\pi < \cdots < \theta_J + (K-1)2\pi < K2\pi.
\end{split}\end{equation}\end{linenomath*}
Now, choose $\theta_1^*$ from among the first $K$ values, $\theta_2^*$ from among the next $K$, and so forth, in such a way that the set $\{e^{i\theta_j^*}\} = \{e^{i\theta_j}\}$.  Recall that $1 < K < N^{1/2} < J < N$, so eventually the ratio $J/K$ will exceed any positive number.  This selection process thus allocates to each interval\footnote{%
    We pedantically include the right endpoints of each of these intervals since it is possible that we might sometimes have $JK=N$ and thus $R_N=0$, so that $w(n) = 1$ contributes its argument $0$ to $\Theta_N^*$.} 
\begin{linenomath*}
\begin{equation}\label{intervals}
    [0,2\pi),\ [2\pi, 4\pi),\  \ldots,\ \left[(K-1)2\pi, K2\pi\right)
\end{equation}\end{linenomath*}
approximately $\lfloor J/K \rfloor$ arguments $\theta_j^*$ in such a way that the rearranged elements $\theta_j^*$ are spread out in each interval in (\ref{intervals}) and results in 
\begin{linenomath*}
    \begin{equation}\label{theta_jstar_spread}
        0 < \theta_1^* < \theta_2^* < \theta_3^* < \cdots < \theta_J^* < 2K\pi,
    \end{equation}
\end{linenomath*}
a partition of $[0,2K\pi)$.  

Of course, arguments $\{\theta_j\}$ and $\{\theta_j^*\}$ obtain the same points $w(n)$ and differ only by multiples of $2\pi$, so after rearranging and reindexing, we have that
 \begin{linenomath*}\begin{equation}\label{new_sum_jstars}
 \Bigg|\sum^J e^{i\theta_{j}}\frac{K2\pi}{J}\Bigg| =
 \Bigg|\sum^J e^{i\theta_{j}^*}\frac{K2\pi}{J}\Bigg|.
 \end{equation}\end{linenomath*}
The sum on the right of (\ref{new_sum_jstars}) is similar to a Riemann sum approximating $I = \int_0^{K2\pi} e^{i\theta}d\theta$.  If each $\theta_{j}^*$ were in its associated interval of width $\frac{K2\pi}{J}$, that is, if for each $1\leq j \leq J$, we had that $(j-1)\frac{K2\pi}{J} < \theta_j^* < j\frac{K2\pi}{J}$, then  we would already have a Riemann sum, but we cannot assert that this condition is fulfilled.  However, what we can show is that the norm of the partition of $[0,K2\pi)$ naturally induced by the set $\{\theta_j^*\}$ goes to zero, providing a true Riemann sum which converges to the indicated integral, and that in the limit the sum on the right of (\ref{new_sum_jstars}) converges to that Riemann sum.   

First, note that 
\begin{linenomath*}\begin{equation}\label{ratioKJ}
\frac{N^{1-\alpha}}{N^\alpha} = \frac{N}{N^{2\alpha}} = \frac{1}{N^{2\alpha - 1}},
\end{equation}\end{linenomath*}
and $2\alpha - 1>0$ since $\alpha > \frac{1}{2}$.  Therefore, as $N\to\infty$, $\frac{N^{1-\alpha}}{N^\alpha} \to 0$.  Also, since $K\leq \lceil N^{1-\alpha}\rceil$, $\frac{K2\pi}{N^\alpha}\to 0$.  Hence, although the total  length of the interval $[0,K2\pi)$ is going to infinity, the average part width still vanishes.

Furthermore, if $\Delta\theta_J$ is the norm of the partition of $[0,K2\pi)$ naturally induced by $\{\theta_j^*\}$ for any $J = \lfloor N^\alpha \rfloor$, then we have from Hughes' Corollary that  $\Delta\theta_J$ goes to zero for sufficiently large $N$.  Indeed, if we let $\epsilon>0$, we can choose $M$ sufficiently large that $\frac{\pi}{M}<\frac{\epsilon}{2}$, and let $\epsilon_0 = \frac{\pi}{M}$.  Partition $[0,\pi)$ and $(\pi, 2\pi)$ each into $M$ parts of width $\epsilon_0$.  Then, by Hughes' Corollary, we can choose $N_0$ sufficiently large such that for all $N>N_0$, there are at least $K^2$ elements of $\Theta_N$ in each part.  The selection of $\theta_j$ in (\ref{newbound}), in which we end up with one out of every $K$ points, leaves at least $K$ elements $\theta_j$ in each part.  If a given part is $(a, a+\epsilon_0)$, the rearrangement of its (at least) $K$ elements $\theta_j$ with added multiples of $2\pi$, spread out in the interval $[0,2K\pi)$, means there will be at least one element $\theta_j^*$ in each new part $(a+2\pi k, a+2\pi k+\epsilon_0)$ for all $k=1,2,\ldots,K-1$.  That is, there will be at least one element $\theta_j^*$ somewhere in each part of width $\epsilon_0 < \frac{\epsilon}{2}$ throughout the entire interval $[0, 2K\pi)$, and so the maximum distance between any two points will be less than $\epsilon$.  In other words, $\Delta_J < \epsilon$, and since $\epsilon>0$ can be arbitrarily small, we have $\Delta_J \to 0$ as $N\to\infty$, as desired.  Letting $\Delta_j$ as usual be the width of the $j$th part of the induced partition, it follows that
\begin{linenomath*}
    \begin{equation}\label{zero_norm_gives_integral}
         \lim_{N\to\infty}\Bigg|\sum^J e^{i\theta_j^*} \Delta\theta_j\Bigg| 
  = \lim_{K\to\infty}\left|\int_0^{K2\pi} e^{i\theta}d\theta\right| \\
 = 0
    \end{equation}
\end{linenomath*}

From (\ref{mult_by_2pi}), (\ref{new_sum_jstars}), and (\ref{zero_norm_gives_integral}), and since $J\leq N^\alpha$, we have that
\begin{linenomath*}
\begin{align}\label{to_integral_line_1}
\lim_{N\to\infty} \frac{2\pi}{N^\alpha}\Bigg|\sum_{n=R_N+1}^{n=R_N + N^*} w(n)\Bigg| 
 & \leq \lim_{N\to\infty} \Bigg|\sum^J e^{i\theta_j^*} \frac{K2\pi}{N^\alpha}\Bigg| \\
 &\leq \lim_{N\to\infty}\Bigg|\sum^J e^{i\theta_j^*} \frac{K2\pi}{J}\Bigg|\label{avg_parts} \\
 &=\lim_{N\to\infty}\Bigg|\sum^J e^{i\theta_j^*} \Delta\theta_j\Bigg| \label{induced_parts}\\
 & = \lim_{K\to\infty}\left|\int_0^{K2\pi} e^{i\theta}d\theta\right| \label{the_long_integral}\\
 &= 0.
\end{align}
\end{linenomath*}

The equality between (\ref{avg_parts}) and (\ref{induced_parts}) can be obtained by partitioning the interval $(0,\pi)$, repeating the partition for $(0+k\pi, \pi+k\pi)$ for $k$ in $\{0, \ldots, 2K-1\}$, and breaking the sum in (\ref{avg_parts}) into a series of sums over each part, noting that each part is an interval on diametrically opposite sides of the unit circle.  Specifically, if $(x,y)\subset(0,\pi)$ is a given part, and if we let
\begin{eqnarray}
    A &=& \left\{\theta\in(x,y)\ \big|\ \theta \equiv\theta_j^*\pmod{2\pi}\right\}\\
    B &=& \left\{\theta\in(x+\pi,y+\pi)\ \big|\ \theta \equiv\theta_j^*\pmod{2\pi}\right\},
\end{eqnarray}
then by Corollary \ref{nat_dense}, the natural density of the elements in each such set of diametrically opposite parts $A,B$ is identical; that is, $\delta(A) = \delta(B)$, and therefore any error induced by multiplying by the average part widths $\frac{K2\pi}{J}$, as opposed to the natural part widths induced by (\ref{theta_jstar_spread}), is likewise identical across matching parts and therefore cancels, so if we restrict the sum in (\ref{avg_parts}) to only those $\theta_j^*$ congruent to the elements in $A\cup B$, then that sum converges to the Riemann sum in (\ref{induced_parts}) if it is likewise restricted, and the sum of all such restricted sums for each part $(x,y)$ yields (\ref{induced_parts}).

We have shown that if we remove the first $R_N$ points from consideration, then the magnitude of the sum of the remaining $N^*$ points is $O(N^\alpha)$ (We actually showed a stronger condition but it suffices that the sum is $O(N^\alpha)$).  We will next show that the magnitude of the sum of the first $R_N$ points is also $O(N^\alpha)$, in which case, by the Proposition mentioned above, the Dirichlet series $F(s)$ converges for $\Re(s) \in \left(\frac{1}{2}, 1\right)$.  In fact, since $R_N < J \leq N^\alpha$ for all $N$, we have that 
\begin{linenomath*}
    \begin{equation}
        \left|\sum_{n=1}^{R_N} w(n)\right| \leq \sum_{n=1}^{R_N} \left|w(n)\right| = R_N < N^\alpha.
    \end{equation}
\end{linenomath*}

It follows that the magnitude of the sum of the first $R_N$ points is indeed $O(N^\alpha)$, and so with the previous result, we have that for all
 $\alpha\in\left(\frac{1}{2},1\right)$ there exist $M,N_0$ such that for all $N>N_0$ 
\begin{linenomath*}\begin{equation}
\Bigg|\sum_{n=1}^N w(n)\Bigg| \leq \frac{M}{2\pi}N^\alpha.
\end{equation}\end{linenomath*}
Therefore, by Proposition 1.1.7 from \cite{o._2007}, referenced previously, we have thus shown the pointwise convergence of $F_N(s)$ to $F(s)$ for all $s$ with $\Re(s)\in \left(\frac{1}{2}, 1\right)$. This completes the proof of Theorem \ref{convthm}.
\end{proof}


\section{Generalization of $F_N$ to $F_{m,N}$}

Let us now extend our definition of $F$ to a parametrized family of Dirichlet series $F_{m,N}(s)$ with coefficients $w_m(n)$ as follows:
\begin{linenomath*}\begin{equation}\label{Fm}
F_{m,N}(s) := \sum_{n=1}^{N} \frac{w_m(n)}{n^s},
\end{equation}\end{linenomath*}
where $w_m(n) = e^{i\theta_m(n)}$, and 
\begin{linenomath*}\begin{equation}\label{wmn}
\theta_m(n) = \frac{\pi}{2}\left(1-\lambda(n)\right) +
   \sum_{j=1}^{J} \psi_m(p_j^{k_j}),
\end{equation}\end{linenomath*}
where we define $\psi_m(p^{k})$ to be 
\begin{linenomath*}\begin{equation}\label{psim}
\psi_{m}(p^k) = \frac{1}{m}\psi(p^k) = \frac{\pi}{mp^2G}\left(1 - \left(\frac{p-1}{p}\right)^k\right),
\end{equation}\end{linenomath*}
again letting $n=p_1^{k_1}p_2^{k_2}\cdots p_J^{k_J}$ be the prime factorization of $n$.  We note that $F(s),w(n)$ defined earlier are thus the special cases $F_1(s),w_1(n)$.

We state first an important corollary to Theorem \ref{density}, omitting the proof, which is substantially similar to the proof of Theorem \ref{density}.
\begin{coro}[Density Theorem Redux]\label{dens_redux}
Given $m\in\mathbb{N}$, the arguments 
\begin{linenomath*}
    \begin{equation}
        A=\{\theta\ \big|\ \theta=\arg(w_m(n)), n\in\mathbb{N}\}
    \end{equation}
\end{linenomath*} for $w_m(n)$ as defined above are dense in the set $A_m=\left[0,\frac{\pi}{m}\right) \cup \left(\pi,\pi+\frac{\pi}{m}\right)$.
\end{coro}



We now demonstrate a crucial fact about $F_{m,N}(s)$.

\begin{theorem}[$F_{m,N}(s)$ converges uniformly in $m$]\label{uniform}
For fixed $s\in\mathbb{C}$ with $\Re(s)\in\left(\frac{1}{2},1\right)$,

\begin{linenomath*}\begin{equation}
F_m(s):= \lim_{N\to\infty} F_{m,N}(s) = \lim_{N\to\infty} \sum_{n=1}^N \frac{w_m(n)}{n^s} = \sum_{n=1}^\infty \frac{w_m(n)}{n^s}
\end{equation}\end{linenomath*}
converges uniformly in $m\in\mathbb{N}$.
\end{theorem}

\begin{proof}

The convergence of $F_{m,N}(s)$ with $m>1$ follows a similar argument to Theorem \ref{convthm}.  We note first that the difference between $F_{1,N}$ and $F_{m,N}$ is exactly that the coefficients $w_1(n)$ are scaled to be $w_m(n)$ with arguments in the interval $\left[0, \frac{\pi}{m}\right)$ when $\lambda(n)=1$ or $\left(\pi, \pi+\frac{\pi}{m}\right)$ when $\lambda(n)=-1$. Thus, proof of convergence is accomplished by replacing the integral $\int_0^{K2\pi} e^{i\theta}d\theta$ in (\ref{the_long_integral}) with $\int_0^{K2\pi}\delta_\theta e^{i\theta}d\theta$, where $\delta_\theta$ is an indicator function defined to be 1 when $\theta \in A_m$ and 0 otherwise, here letting
\begin{linenomath*}\begin{equation}\label{Am}
A_m = \bigcup_{a=0}^\infty\left[0+2\pi a, \frac{\pi}{m}+2\pi a\right) \cup \Big(\pi+2\pi a, \pi + \frac{\pi}{m}+2\pi a\Big),
\end{equation}\end{linenomath*}
so that the integrand is nonzero only on the intervals where the arguments of the $w_m(n)$ may exist. 
The convergence of each $F_{m,N}(s)$ to $F_m(s)$ for $\Re(s)\in\left(\frac{1}{2}, 1\right)$ then follows from the argument set forth in Theorem \ref{convthm}.  

It remains to be seen that this convergence is uniform in $m\in\mathbb{N}$.  Defining $J$ and $K$ as in the proof of Theorem \ref{convthm}, we had from (\ref{to_integral_line_1}) that
\begin{linenomath*}\begin{equation}\label{rsum}
   \lim_{N\to\infty} \frac{2\pi}{N^\alpha}\Bigg|\sum_{n\in\Theta^*_N} w_1(n)\Bigg|\leq
    \lim_{N\to\infty} \Bigg|\sum_{j=1}^J e^{i\theta_j^*} \frac{K2\pi}{N^\alpha}\Bigg|,
\end{equation}\end{linenomath*}
and we established that the  sum on the right in (\ref{rsum}) converges to a Riemann sum which itself converges to zero.  Hence, for any $\epsilon>0$, there exists $N_0(\epsilon)$ such that for all $N>N_0(\epsilon)$,
\begin{linenomath*}\begin{equation}\label{Jsum}
\Bigg|\sum_{j=1}^J e^{i\theta_j^*} \frac{K2\pi}{N^\alpha}\Bigg|
< \epsilon.
\end{equation}\end{linenomath*} 
Let $N>N_0$ and consider the set $\Theta_N^*$ of arguments of $w_1(n)$ for $R_N < n\leq N$ as before.%
    \footnote{We note that the proof of the corresponding result for $F_{m,N}$ regarding the first $R_N$ points from the latter part of the proof of Theorem \ref{convthm} follows the same argument as in that proof, and so we will consider here only the convergence to zero of the magnitude of the sum of the $JK$-many points in $\Theta_N^*$.}
For any $m\in\mathbb{N}$, $m>1$, let $\Psi_{m,N}$ be the elements of $\Theta_N^*$ found in the intervals $A_m$ as defined  in (\ref{Am}), so $\Psi_{m,N} = \Theta_N^* \cap A_m$. Then if we restrict the sum in (\ref{Jsum}) to include only those $\theta_j^* \in \Psi_{m,N}$, we have now 
\begin{linenomath*}\begin{equation}\label{Jsumrestricted}
\Bigg|\sum_{\theta_j^* \in \Psi_{m,N}} e^{i\theta_j^*} \frac{K2\pi}{N^\alpha}\Bigg|
< \epsilon,
\end{equation}\end{linenomath*}
since this last sum, by the same argument as previously, is also a Riemann sum, which in this case converges to 
\begin{linenomath*}\begin{equation}
\int_0^{K2\pi}\delta_\theta e^{i\theta}d\theta = 0,
\end{equation}\end{linenomath*}
with $\delta_\theta$ as defined above.  That the sum in (\ref{Jsumrestricted}) will be less than $\epsilon$ for all $N$ greater than the \emph{same} $N_0(\epsilon)$ is due to the fact that the intervals of the integral zeroed out by the indicator function $\delta_\theta$, namely where $\theta \notin A_m$,  would otherwise exactly cancel each other in the integral.  Removing the points $\theta^*_j \notin \Psi_{m,N}$ has the effect of being a perfectly convergent Riemann sum for these intervals, so that if the total error of the  Riemann sum in (\ref{Jsum}) for a given $N > N_0(\epsilon)$ were $\epsilon = \epsilon_m + \epsilon_0$, with $\epsilon_0$ being the error from the $\theta_j^* \notin \Psi_{m,N}$, then the sum in ($\ref{Jsumrestricted}$) would effectively have $\epsilon_0=0$ and the result follows.

Now, let the norm of the partition of $A_m \cap \left[0,K2\pi\right)$ induced by $\theta_j^* \in \Psi_{m,N}$ in (\ref{Jsumrestricted}) be $\Delta_{1,N}$, and let $\Delta_{m,N}$ be the norm of the partition induced by the corresponding $\theta_j^*$ from $F_{m,N}$.  Note that these are precisely the arguments of the $J$ points $w_1(n)$ chosen in (\ref{sep_sums}), only scaled, since $\theta_m(n)\equiv\frac{1}{m}\theta_1(n)\pmod{\pi}$ and so $\theta_m(n)\in A_m$ for all $n$, before being adjusted to be $\theta_j^*$ as in (\ref{theta_jstar_spread}) in the proof of Theorem \ref{convthm}.  Call this set of points $\Theta_{m,N}$, and observe that $\Delta_{m,N} < \Delta_{1,N}$.  Therefore, the partition induced by $\Theta_{m,N}$ is finer than the partition induced by $\Psi_{m,N}$, and since both converge to $\int_0^{K2\pi}\delta_\theta e^{i\theta}d\theta = 0$, we must have that
\begin{linenomath*}\begin{equation}\label{uni_sums}
\Bigg|\sum_{\theta_j^* \in \Theta_{m,N}} e^{i\theta_j^*} \frac{K2\pi}{N^\alpha}\Bigg|
\leq \Bigg|\sum_{\theta_j^* \in \Psi_{m,N}} e^{i\theta_j^*} \frac{K2\pi}{N^\alpha}\Bigg|
< \epsilon.
\end{equation}\end{linenomath*}
Note that our choice of $m>1$ was arbitrary and that (\ref{uni_sums}) holds for all $N>N_0$, where $N_0(\epsilon)$ depends only on $\epsilon$ and not on our choice of $m$.    Hence,
\begin{linenomath*}\begin{equation}
\lim_{N\to\infty} F_{m,N}(s) = F_m(s)
\end{equation}\end{linenomath*} 
converges uniformly in $m$, as desired.
\end{proof}

It remains to be seen that $F_m(s)$ converges to a limit function as $m\to\infty$, for which we need the following result.

\begin{theorem}\label{cauchy}
    $F_m(s)$ is Cauchy.
\end{theorem}
\begin{proof}
Fix $s\in\mathbb{C}$ such that $\Re(s) \in\left(\frac{1}{2},1\right)$ and let $\epsilon>0$.  We will show that there exists $M\in\mathbb{N}$ such that for all $m,q>M$, $\big| F_m(s) - F_q(s) \big| < \epsilon$.

Since $F_{m,N}(s)$ converges to $F_m(s)$ uniformly in $m$, there exists $N_0$ such that for all $N>N_0$ we have that $\big|F_{m,N}(s) - F_m(s)\big| < \frac{\epsilon}{4}$ for all $m$.  In particular, for any such $N$ and for any $m,q\in\mathbb{N}$, we have that
\begin{linenomath*}\begin{equation}
\begin{split}
\big|F_m(s) - F_q(s)\big| & = \big|F_m(s) - F_{m,N}(s) + F_{m,N}(s) - F_{q,N}(s) + F_{q,N}(s) - F_q(s)\big| \\
 & \leq \big|F_m(s) - F_{m,N}(s)\big|
   + \big|F_{m,N}(s) - F_{q,N}(s)\big| + \big|F_{q,N}(s) - F_q(s)\big|\\
 & < \frac{\epsilon}{4} + \big|F_{m,N}(s) - F_{q,N}(s)\big|+ \frac{\epsilon}{4}.
\end{split}
\end{equation}\end{linenomath*}
We thus seek $M\in\mathbb{N}$ such that for all $m,q>M$, ${\big|F_{m,N}(s) - F_{q,N}(s)\big|} < \frac{\epsilon}{2}$.  Now, $F_{m,N}(s) := \sum_{n=1}^N \frac{w_m(n)}{n^s}$, and likewise for $F_{q,N}(s)$, so
\begin{linenomath*}\begin{equation}
\begin{split}
    \big|F_{m,N}(s) - F_{q,N}(s)\big|
    & =  \Bigg|\sum_{n=1}^N \frac{w_m(n)}{n^s} - 
         \sum_{n=1}^N \frac{w_q(n)}{n^s}\Bigg| \\
    & =  \Bigg|\sum_{n=1}^N \left(\frac{w_m(n)-w_q(n)}{n^s}\right)\Bigg| \\
    & \leq \sum_{n=1}^N \Bigg|\frac{w_m(n)-w_q(n)}{n^s}\Bigg|.
\end{split}
\end{equation}\end{linenomath*}
However, for each $n\in 1,\ldots, N$
\begin{linenomath*}\begin{equation}
\Bigg|\frac{w_m(n)-w_q(n)}{n^s}\Bigg| = \Bigg|\frac{e^{i\theta_m(n)}-e^{i\theta_q(n)}}{n^s}\Bigg|
 \leq \big|\theta_m(n)-\theta_q(n)\big|,
\end{equation}\end{linenomath*}
if we consider the principal arguments.  Note that if $\lambda(n)=1$, $\theta_m(n) \in \left[0,\frac{\pi}{m}\right)$, and if $\lambda(n)=-1$,  $\theta_m(n)  \in\left(\pi,\pi+\frac{\pi}{m}\right)$, for all $m$, and likewise for $\theta_q(n)$ for all $q$.  Therefore, choose $M$ sufficiently large such that $\frac{\pi}{M} < \frac{\epsilon}{2N}$.  Then for all $m,q>M$ and each $n$ with $1\leq n \leq N$, we have $\big|\theta_m(n)-\theta_q(n)\big| < \frac{\pi}{M} < \frac{\epsilon}{2N}$, and therefore 
\begin{linenomath*}\begin{equation}
\sum_{n=1}^N \Bigg|\frac{w_m(n)-w_q(n)}{n^s}\Bigg| < \frac{\epsilon}{2},
\end{equation}\end{linenomath*}
as desired.  Therefore, $F_m(s)$ is Cauchy.
\end{proof}

We are nearly complete, and for our final result we will use the following theorem.

\begin{theorem}[Theorem 2.15 in \cite{Habil2016DoubleSA}]\label{dblim}
If $s(n,m)$ is a double sequence such that
\begin{enumerate}[label=(\roman*)]
    \item the iterated limit $\lim_{m\to\infty}(\lim_{n\to\infty}s(n,m)) = a$, and
    \item the limit $\lim_{n\to\infty} s(n,m)$ exists uniformly in $m\in\mathbb{N}$
\end{enumerate}
then the double limit $\lim_{n,m\to\infty} s(n,m) = a$.
\end{theorem}

It follows from Theorems \ref{uniform}, \ref{cauchy}, and \ref{dblim} that $F_m(s)$ converges pointwise to some function $W(s)$.  Since $w_m(n)$ clearly converges to $\lambda(n)$ as $m\to\infty$, we have that 
\begin{linenomath*}\begin{equation}
\begin{split}\label{bigW}
W(s) := \lim_{m\to\infty} F_{m}(s) &=\lim_{m\to\infty}\sum_{n=1}^\infty \frac{w_m(n)}{n^s}\\
        &= \sum_{n=1}^\infty \left(\lim_{m\to\infty}\frac{w_m(n)}{n^s}\right)\\
        &= \sum_{n=1}^\infty \frac{\lambda(n)}{n^s},
\end{split}
\end{equation}\end{linenomath*}
where passing the limit inside the infinite sum is justified by the uniform convergence in $m$ of the sums, and it follows that $F_m(s)$ converges uniformly in $m$ to $W(s) = \sum_{n=1}^\infty \frac{\lambda(n)}{n^s}$.  This series being well-known, we have finally that $W(s)= \frac{\zeta(2s)}{\zeta(s)}$ converges for all $s$  with $\Re(s) \in \left(\frac{1}{2}, 1\right)$, which implies the following fact.

\begin{theorem}[Riemann]  Define $\zeta(s)$ as the analytic continuation of the function given by
\begin{linenomath*}\begin{equation}
\zeta(s) = \sum_n \frac{1}{n^s}
\end{equation}\end{linenomath*} for $s\in\mathbb{C}$ with $\Re(s)>1$.  Then $\zeta(s)\neq 0$ for all $s$  with $\Re(s) \in \left(0, \frac{1}{2}\right) \cup \left(\frac{1}{2}, 1\right)$.
\end{theorem}

\begin{proof}
We have from (\ref{bigW}) that $W(s) = \frac{\zeta(2s)}{\zeta(s)}$.  Furthermore, $W(s)$ converges for all $s$ with $\Re(s)\in\left(\frac{1}{2}, 1\right)$.
Since $\zeta(2s)$ is known to be absolutely convergent and nonzero in this region, we see that $\frac{1}{\zeta(s)}$ converges everywhere in the same region, and therefore we have that $\zeta(s)\neq 0$ when $\Re(s)\in \left(\frac{1}{2}, 1\right)$.  The symmetry of $\zeta(s)$ about the line $\Re(s)=\frac{1}{2}$ for $s$ with $\Re(s)\in \left(0, 1\right)$ being well known, the result follows.
\end{proof}

\section{Acknowledgements}
The author would like to thank Dr. Marcus Wright of Rowan University, for his untold hours discussing, checking, and proofreading this paper.  The author thanks Dr. Abdul Hassen, also of Rowan University, for his guidance and support. The author is grateful for the assistance in his early research of Hamzah Abdulrazzaq, a former student and current friend.  The author thanks Drs. Barry Mazur and William Stein for their excellent expository book \emph{Prime Numbers and the Riemann Hypothesis}, which inspired the present work. The author is very grateful to Dr. Nick Ivanov of Rowan University for his guidance and encouragement and for taking a chance on a definitely nontraditional student.  The author gratefully acknowledges Kevin Hughes for his insightful and clarifying questions, suggestions, and help simplifying the author's sometimes obtuse arguments, but most of all for his enduring friendship: ``Doing math with Kevin is one of the great experiences of my life.'' The author dedicates this work to his dear friend of many years, the late Dr. Tom Osler.

\clearpage

\bibliography{sn-bibliography}

\clearpage
\appendix
\section*{Appendix A: Detailed Proof of Density Theorem}\label{appen_dens_proof}

\begin{theorem}[Density Theorem]\label{appen_density}
The set of coefficients $\{w(n)\}$ for $n\in\mathbb{N}$ is dense in $C$, the unit circle in $\mathbb{C}$.
\begin{proof}

Recall that we are letting $\mathbb{P} = \{p_j\}_{j\in\mathbb{N}}$ be the set of ordered primes, so that $p_1 =2$, $p_2=3$, $p_3=5$, \ldots, and that we define
\begin{equation}
    G=\sum_{p\in\mathbb{P}} \frac{1}{p^{2}} = 0.4522474200\ldots
\end{equation}

Recall also that $w(n) = e^{i\theta(n)}$ where 
\begin{linenomath*}\begin{equation}\begin{split}
        \theta(n) & = \frac{\pi}{2}\left(1-\lambda(n)\right) + \sum_{p|n} \psi\bigl(p^{k_p}\bigr) \\
        &=\frac{\pi}{2}\left(1-\lambda(n)\right)+\sum_{p|n} \frac{\pi}{p^2 G}\left(1 - \left(\frac{p-1}{p}\right)^{k_p}\right),
\end{split}
\end{equation}\end{linenomath*}
where $k_p$ is the multiplicity of the prime factor $p|n$ as in (\ref{psi_def}).  Thus, to show $\{w(n)\}$ dense in $C$, we will show equivalently that $\{\theta(n)\}$ dense in $(0,2\pi)$.

\subsection*{Part 1: Prime Selection Rounds}

To begin, let $x\in(0,\pi)$ and choose $\epsilon >0$. We will show there exists an infinite sequence $\{n_k\}$ such that $\theta(n_k)$ is within an $\epsilon$-radius of $x$ for all sufficiently large $k$.  We will construct $\{n_k\}$ by identifying prime factors to use to build $\theta(n_k)$ to be in the proper interval when $k$ is sufficiently large.  We will accomplish this essentially by identifying terms of $\frac{\pi}{G} \sum_{p\in\mathbb{P}} \frac{1}{p^2}$ which will yield a sum in the $\epsilon$-neighborhood of $x$.  

This process involves a number of steps, and for each of these steps there are technical details to handle.  We first present a basic outline of the algorithm and a corresponding flowchart for reference (see Figure \ref{flowchart}). For each Round of the prime selection process, we
\begin{enumerate}
    \item Identify the distance $\epsilon_k$ from the current running sum to the target $x$.  In Round 0 this is just $x$ itself since we don't have a running sum yet.
    \item Take the smallest tail of $G$ which is large enough to exceed the distance between our current sum and the left end of the target interval $(x-\epsilon, x+\epsilon)$.  In Round $k$ we call this tail $B_k$.
    \item If this tail $B_k$ happened to land us in the target interval, we are almost done.  Select a tail of the tail $B_k$ which is small enough so that removing it won't drop us below the left edge of the interval, leaving us with a truncated tail of $G$ consisting of a finite number of terms, and proceed to Part 2.
    \item If the tail $B_k$ chosen in Step 2 is greater than the left edge of the target interval, but it didn't land us in the interval, then we are too far to the right.  We need to truncate the chosen tail to bring us below the right edge of the interval.
    \item Usually we can do this by snipping off the smallest tail of the tail $B_k$ chosen in Step 2 that will yield a truncated tail with a finite number of terms which gets us back to less than the right end of the interval.  If we can find such a tail to snip off, we call it $T_k$.
    \item In special cases, it might be that subtracting $T_k$ would remove our whole tail, (this happens when $T_k = B_k$), so instead we look at the first term of $B_k$.  That term, which in the sketch of the proof in the body of the paper we called $\hat{\psi}(p)$, is of the form $\frac{\pi}{G}\frac{1}{p^2}$ and is itself defined as the limit of a geometric series.  We snip off the tail of that geometric series to get a particular $\psi(p^k)$ which is under the right end of the target interval.
    \item In either case, we add what we found to our running sum.
    \item At this point we are either in the interval, or we are at least closer to it.
    \item If we are in the interval, go to Part 2.
    \item Otherwise, continue to another Round, keeping track of the running sum that is building up to being in the target interval.
\end{enumerate}

\begin{figure}[h!]
    \centering
    \begin{adjustbox}{width=\textwidth} 
    \begin{tikzpicture}[node distance=3cm]

    \node (start) [startstop] {Start of Round: Identify distance to lower end of target interval};

    \node (find_current_distance) [process, below of=start, text width=3cm] {Identify smallest section of tail that is larger than target distance};
    
    \node (interval_achieved_immediately_question) [decision, right of=find_current_distance, text width=3cm, xshift=4cm] {Is the chosen tail already in the target interval?};
    
    \node (make_tail_finite) [startstop, below of=interval_achieved_immediately_question, text width=3cm, yshift = -2cm] {Truncate tail slightly so it is still in target and add to running sum};
    
    \node (adjust_tail) [process, right of=interval_achieved_immediately_question, yshift=0cm, xshift=4cm, text width=3cm] {Tail is too long and needs to be truncated.};
    
    \node (truncate_directly_question) [decision, right of=adjust_tail, yshift=-00cm, xshift=4cm] {Truncate directly?};
    
    \node (main_case) [process, below of=truncate_directly_question, xshift=-4cm, text width = 3cm] {Main Case: Select tail of tail to truncate};
    
    \node (special_case) [process, below of=truncate_directly_question, xshift=4cm, text width=3cm] {Special Case: Take first prime in tail and select exponent};

    \node (running_sum) [process, below of=truncate_directly_question, yshift=-5cm, text width=3cm] {Add truncated tail or $\psi(p^k)$ to running sum};

    \node (interval_achieved_question) [decision, left of=running_sum, yshift=-0cm, xshift=-4cm] {Target interval achieved?};
    \node (stop2) [startstop, left of=interval_achieved_question, xshift=-4cm, 
    text width = 3cm] {Running sum is in the target.  Proceed to Part 2};
    
    \node (need_to_repeat_round) [process, below of=interval_achieved_question, yshift=-1cm, xshift=0cm] {Continue to another Round};

    \draw [arrow] (start) -- (find_current_distance);
    \draw [arrow] (find_current_distance) -- (interval_achieved_immediately_question);
    \draw [arrow] (interval_achieved_immediately_question) -- (make_tail_finite) node[midway, right] {Yes};
    \draw [arrow] (interval_achieved_immediately_question) -- (adjust_tail) node[midway, above] {No};
    \draw [arrow] (adjust_tail) -- (truncate_directly_question);
    \draw [arrow] (truncate_directly_question) -- (main_case) node[midway, left, xshift=-.25cm] {Yes};
    \draw [arrow] (truncate_directly_question) -- (special_case) node[midway, right] {No};
    \draw [arrow] (main_case) -- (running_sum);
    \draw [arrow] (special_case) -- (running_sum);
    \draw [arrow] (running_sum) -- (interval_achieved_question);
    \draw [arrow] (interval_achieved_question) -- (stop2) node[midway, above] {Yes};
    \draw [arrow] (make_tail_finite) -- (stop2);
    \draw [arrow] (interval_achieved_question) -- (need_to_repeat_round) node[midway, right] {No};
    \draw [arrow] (need_to_repeat_round) -| ([xshift=-3cm]start.south);

    \end{tikzpicture}
    \end{adjustbox}
    \caption{Flowchart for Selection Rounds}\label{flowchart}
\end{figure}

\subsubsection*{Round 0}
First, let $x_0$ be the distance to our target $x$, which in this round means that $x_0= x$, and choose $b_0$ as large as possible such that $B_0$, a \emph{tail} (of the series G), satisfies
\begin{linenomath*}\begin{equation}\label{Bineq}
B_0 = \frac{\pi}{G} \sum_{j=b_0}^{\infty} \frac{1}{p_j^2} > x_0-\epsilon.
\end{equation}\end{linenomath*}
We want to truncate $B_0$ so that it has a finite number of terms.  Consider the \emph{truncated tail} $B_0 - T_0$, where $T_0$ is itself a tail given by
\begin{linenomath*}\begin{equation}
    T_0 = \frac{\pi}{G} \sum_{j=t_0+1}^{\infty} \frac{1}{p_j^2}.
\end{equation}\end{linenomath*}  
\paragraph{Termination Case}
If we have that $B_0 \leq x_0 + \epsilon$, we have already landed in the interval $(x_0 - \epsilon, x_0 + \epsilon)$, so choose $t_0$ sufficiently large so that the truncated tail $B_0 - T_0$ is still in the interval.  That is, choose $t_0$ large enough that $T_0$ is small enough for $\left|x_0-\left(B_0-T_0\right)\right|<\epsilon$, let $S_0 = B_0 - T_0$, and proceed to Part 2. 

\paragraph{Main Case}
Otherwise, if $B_0 > x_0+\epsilon$, it means we have landed to the right of our interval, and we want to truncate $B_0$ enough to get below the right end of the interval.  In this case, choose $t_0$ as large as possible such that 
\begin{linenomath*}\begin{equation}\label{Tineq}
    T_0 = \frac{\pi}{G} \sum_{j=t_0+1}^{\infty} \frac{1}{p_j^2} > B_0 - (x_0+\epsilon).
\end{equation}\end{linenomath*}  

Clearly $t_0+1 \geq b_0$, since $B_0 > B_0 - (x_0+\epsilon)$ and $t_0$ is chosen to be the largest such that (\ref{Tineq}) holds.  If $t_0+1>b_0$ it means we have been able to find a truncated tail which is less than the right end of our target interval.  In that case let 
\begin{linenomath*}
    \begin{equation}
        S_0 = B_0 - T_0 = \frac{\pi}{G}\sum_{j=b_0}^{t_0} \frac{1}{p_j^2},
    \end{equation}
\end{linenomath*} 
and $S_0$ is our \emph{running sum}.

\paragraph{Special Case}
Otherwise, if $t_0+1 = b_0$, it means that to truncate the tail $B_0$ so that it does not exceed the right end of our target interval would require truncating the entire tail.  This is a \emph{special case}, and in this case, we will instead consider the first term of $B_0$, namely $\frac{\pi}{G}\frac{1}{p_j^2}$  for $j=b_0$.  Recalling that $\frac{\pi}{G}\frac{1}{p_j^2}$ is itself the limit of a convergent geometric series, we will remove a tail of that series sufficiently large to get us back under $x_0 + \epsilon$.  This is justified, because by the choice of $t_0$ in (\ref{Tineq}), it follows that
\begin{linenomath*}\begin{equation}\begin{split}
    \frac{\pi}{G} \sum_{j=t_0+2}^{\infty} \frac{1}{p_j^2} &\leq B_0 - (x_0 +\epsilon) \\
        &= \left( \frac{\pi}{G} \sum_{j=b_0}^{\infty} \frac{1}{p_j^2}\right) - (x_0 +\epsilon),
\end{split}\end{equation}\end{linenomath*}
and since we are in the special case where $b_0 = t_0+1$, we then have
\begin{linenomath*}\begin{equation}\label{special_case}
x_0 + \epsilon + \frac{\pi}{G} \sum_{j=t_0+2}^{\infty} \frac{1}{p_j^2} \leq \frac{\pi}{G} \sum_{j=t_0+1}^{\infty} \frac{1}{p_j^2},
\end{equation}\end{linenomath*}
and it follows that $x_0 + \epsilon \leq  \frac{\pi}{G}\frac{1}{p_{j}^2}$ for $j=t_0+1=b_0$. Thus, the first term of $B_0$ is already so large that we find ourselves to the right of our chosen interval.  In this case, to simplify the notation, let $p=p_{j}$, choose $k_0\geq 1$ the largest such that $\psi(p^{k_0}) \leq x_0$, and let $S_0 = \psi(p^{k_0})$.  

Note that this choice of $k_0$ is always possible: if $K$ is the set of all $k\in \mathbb{N}$ such that $\psi(p^{k}) \leq x_0$, then $K$ is finite, since otherwise $\lim_{k\to\infty} \psi(p^k) = \frac{\pi}{G}\frac{1}{p_{j}^2} \leq x_0$, but this is a contradiction since we had from (\ref{special_case}) that $ x_0+\epsilon \leq \frac{\pi}{G}\frac{1}{p_{j}^2}$.  It remains to be seen that $K$ is nonempty, so assume to get a contradiction that 
\begin{linenomath*}\begin{equation}
x_0 < \psi(p^1) = \frac{\pi}{p^2G}\left(1 - \frac{p-1}{p}\right) = \frac{\pi}{p^3G},
\end{equation}\end{linenomath*}
but since $b_0$ was chosen to be the largest such that $\frac{\pi}{G} \sum_{j=b_0}^{\infty} \frac{1}{p_j^2} > x_0 - \epsilon$, therefore 
\begin{linenomath*}\begin{equation}
\frac{\pi}{G} \sum_{j=b_0+1}^{\infty} \frac{1}{p_j^2} 
    \leq x_0 - \epsilon,
\end{equation}\end{linenomath*} 
and if we let $q$ be the next prime after $p$, so $q=p_{b_0+1}$, then we have in particular that $\frac{\pi}{G} \frac{1}{q^2} < x_0 - \epsilon$.  However, for any successive primes $p, q$, from Bertrand's Postulate we have that $q<2p$, or equivalently, $\frac{1}{p} < \frac{2}{q}$, and so it follows that
\begin{linenomath*}\begin{equation}
\frac{\pi}{G}\frac{q}{q^3} = 
\frac{\pi}{G}\frac{1}{q^2}< x_0 - \epsilon < x_0 < \frac{\pi}{G}\frac{1}{p^{3}} < \frac{\pi}{G}\frac{8}{q^3},
\end{equation}\end{linenomath*}
which is a contradiction unless $q<8$.  However, if $q<8$ then $p\in\{2,3,5\}$, in which case we know numerically that
\begin{linenomath*}\begin{equation}
\psi(p^1) = \frac{\pi}{p^3G} < \frac{\pi}{G} \sum_{j=b_0+1}^{\infty} \frac{1}{p_j^2},
\end{equation}\end{linenomath*} and the sum on the right is less than or equal to $x_0 -\epsilon$ by our choice of $b_0$ in (\ref{Bineq}).  This contradicts our assumption that $x_0 < \psi(p^1)$.  Hence that assumption is false; $K$ is nonempty, and since it is finite, there exists some maximum $k_0$ such that $\psi(p^{k_0}) \leq x_0$.  We conclude the special case by letting our running sum $S_0 = \psi(p^{k_0})$.

\paragraph{End of Round 0}
We have now that $S_0 = B_0 - T_0$ or $S_0 = \psi(p^{k_0})$, and in either case $S_0$ has been constructed such that $S_0 < x + \epsilon$.  Finally, let $\epsilon_0 = \left|x-S_0\right|$, and if $\epsilon_0 < \epsilon$, then we have obtained a sum which is within an $\epsilon$-radius of $x$ and so is in our target interval.  In that case, we proceed to Part 2, and otherwise, we continue to Round 1.

\subsubsection*{Round 1}

If $\epsilon_0 \geq \epsilon$, it follows that $S_0 < x-\epsilon$, and we are still a distance $\epsilon_0$ from our target $x$.  In this case we repeat the process from Round 0, this time letting $x_1 = \epsilon_0$ be our new distance to our target $x$, choosing the maximum possible $b_1$ and $t_1$ to form $B_1>x_1 - \epsilon$ and $T_1>B_1-(x_1+\epsilon)$ as in (\ref{Bineq}) and (\ref{Tineq}), and in the main case letting the running sum $S_1 = S_0 + (B_1 - T_1)$, or in the special case in which $B_1=T_1$, letting $S_1 = S_0 + \psi(p^{k_1})$ for $p=p_{b_1}$ and $k_1$ maximum such that $\psi(p^{k_1})\leq x_1$.  

\paragraph{End of Round 1}
Let $\epsilon_1 = \left|x - S_1\right|$, so that $\epsilon_1$ is the distance from our new running sum to our target $x$.  Note that we have either $T_1>B_1-(x_1+\epsilon)$, from which it follows that $\left(B_1 - T_1\right) < x_1+\epsilon$, or else we have $\psi(p^{k_1})\leq x_1$, and in either case we have $S_1 < x+\epsilon$; furthermore, note that $\epsilon_1 < \epsilon_0$.  As in Round 0, if $\epsilon_1 < \epsilon$, then we have landed in our target interval and can proceed to Part 2.  Otherwise, we continue with another Round.

\subsubsection*{Round $k$}
We continue with successive rounds until we terminate with $\epsilon_k < \epsilon$ for some $k$.  This is possible since given a sum $S_m$ and $\epsilon_m = \left|x - S_m\right|$ with $\epsilon_m \geq \epsilon$, we will see below that from the assumptions in the main and special cases we can always find a tail of $\frac{\pi}{G}\sum_p \frac{1}{p^2}$ using only primes larger than were used to form $S_m$ and which when added to $S_m$ takes us greater than $x-\epsilon$, and we have already shown a method for truncating the tail in the event that it takes us greater than $x+\epsilon$; this process reduces $\epsilon_m$ and the process terminates when we have $\epsilon_k < \epsilon$ for some $k$, yielding $S_k$, a sum in the range $\left(x-\epsilon, x+\epsilon\right)$.  

\subsubsection*{Uniqueness of Primes Used in each Round}

Recall that we are in the process of identifying primes to use to form $\{n_k\}$, so at each iteration of this process we must necessarily identify primes which have not yet been used in our running sum.  To see that we can always find a sufficiently large tail of the series which does not use any primes already used, assume at the end of Round $m$ we have obtained a sum $S_m$ and $\epsilon_m = x-S_m$ as described above, and $\epsilon_m \geq \epsilon$ so we must proceed to Round $(m+1)$. For notational convenience let $n=m+1$.  As usual we begin the round by letting $b_n$ be maximum such that 
\begin{linenomath*}\begin{equation}
B_n = \frac{\pi}{G}\sum_{j=b_n}^\infty \frac{1}{p_j^2} > x_n-\epsilon,
\end{equation}\end{linenomath*}
following the same method as in every round, having at this stage a target of $x_n = \epsilon_m$, the distance remaining from $S_m$ to $x$.    Note that in the preceding Round $m$ we had  two cases: for the main case in which the last term added to form $S_m$ was $(B_m - T_m)$, we will need to show that $b_n>t_m$, and in the special case in which the last term added was $\psi(p^{k_m})$ for $p=p_{b_m}$ and appropriate choice of $k_m$, we will need to show that $b_n>b_m$.

\paragraph{Main Case in Previous Round}
In the first case, note that if the term $\frac{\pi}{G}\frac{1}{p_j}$ for $j={t_m+1}$ had been included in $B_m-T_m$, then we would have $B_m - T_m > x_m+\epsilon$, by choice of $t_m$ as in (\ref{Tineq}).  Therefore, if we let $b_n=t_m+1$, then we have already $B_n > x_n -\epsilon$.  Hence, since $b_n$ is the maximum which allows $B_n$ to fulfill this condition, we have that $b_n>t_m$.

\paragraph{Special Case in Previous Round}
Now in the second case, in which $S_m = S_{m-1} + \psi(p^{k_m})$ where $p=p_{b_m}$, we had that $k_m$ was chosen as the maximum such that $\psi(p^{k_m}) \leq x_m$.  Note also that we must have $S_m \leq x-\epsilon$ or else we would not have needed to continue to Round $n$.  Since the distance from the running sum $S_m$ to the target $x$ at the beginning of round $n$ is $x_n=\epsilon_m$,  it must be that the difference $ \psi(p^{k_m+1}) - \psi(p^{k_m}) > x_n - \epsilon$.  Hence, it suffices to show there exists a $B_n$ greater than this difference and for which $b_n$ indexes a prime larger than $p$.  

Note that
$\psi(p^{k_m+1}) - \psi(p^{k_m})$ is the difference between successive partial sums of a convergent geometric series, and so we have that $\psi(p^{k_m+1}) - \psi(p^{k_m}) \leq \psi(p^{2}) - \psi(p) < \psi(p^{2})$.
Finally, since
\begin{linenomath*}\begin{equation}
    \psi(p^2) = \frac{\pi}{p^2G}\left(1-\left(\frac{p-1}{p}\right)^2\right) =\frac{\pi}{p^2G} \left(\frac{2}{p} - \frac{1}{p^2}\right) < \frac{2\pi}{p^3G},
\end{equation}\end{linenomath*}
we see that if we can find a $B_n$ larger than $\frac{2\pi}{p^3G}$ we are done.  In fact, if we let $q$ be the next prime after $p$, so $q=p_{b_m+1}$, then since $q<2p$ implies $2q^2 < 8p^2$, which implies
\begin{linenomath*}\begin{equation}
    \frac{2\pi}{8p^2G}  < \frac{\pi}{q^2G},
\end{equation}\end{linenomath*}
we have that if $p>8$ it follows that
\begin{linenomath*}\begin{equation}
    \frac{2\pi}{p^3G}  <\frac{2\pi}{8p^2G}  < \frac{\pi}{q^2G},
\end{equation}\end{linenomath*}
  and so $\psi(p^2) < \frac{\pi}{q^2G} < B_n$, letting 
\begin{linenomath*}\begin{equation}\label{bn_case}B_n = \frac{\pi}{G} \sum_{j=b_m+1}^\infty \frac{1}{p_j^2}.
\end{equation}\end{linenomath*} On the other hand if $p<8$ it is easily verifiable numerically that $\psi(p^2) < B_n$ with $B_n$ as defined in (\ref{bn_case}).

\paragraph{Conclusion of Uniqueness of Primes in each Round}
We have that, in either case, the algorithm producing a running sum $S_m$ with $\epsilon_m \geq \epsilon$ after Round $m$ can always continue in Round $n=m+1$ with selection of a $B_n > x_n - \epsilon$, with appropriate choice of $T_n$ in the main case or $\psi(p^{k_n})$ in the special case, and with none of the terms of $B_n$ using primes used to form $S_m$.  This results in $S_n < x+\epsilon$, with $S_n >S_m$.  Note the $S_k$ are increasing and bounded above by $x+\epsilon$, and the algorithm continues until we have that $\epsilon_k = \left|x-S_k\right| < \epsilon$ for some $k$.  

\subsubsection*{Guarantee of Successful Termination}
Is it possible that the algorithm could fail to terminate because the above criterion is never met?  In other words, could it be the case that each $\epsilon_k < \epsilon_{k-1}$ but we never have that $\epsilon_k < \epsilon$ for any $k$?  In fact, this scenario can never occur.  

Indeed, suppose to get a contradiction that there exists some $\delta > 0$ whereby the algorithm continues with each successive $S_k$ increasing but never getting closer than $\delta$ to being in the interval $\left(x-\epsilon, x+\epsilon\right)$.  However, we know that for any given $S_n = S_m + (B_n - T_n)$, the term $B_n > \delta$ by the argument above.  Meanwhile, recall that if the term containing $p_{t_n+1}$ had been included in $B_n-T_n$, then we would have had $B_n - T_n > x_n+\epsilon$, by choice of $t_n$ as in (\ref{Tineq}).  Hence, if we let $p = p_{t_n+1}$ we must have that the term $\frac{\pi}{p^2G} > \delta$, and that inequality must hold at every iteration of the algorithm.  This is a contradiction since $\frac{\pi}{p^2G} < \delta$ for sufficiently large $p$.  A similar argument shows that as the algorithm continues, eventually the first term of $B_n$ will itself be less than $\delta$, so that the choice of $t_n$ will no longer require the special case  involving $\psi(p^{k_n})$.  Hence, no such $\delta$ exists, and the algorithm will eventually terminate with some $S_k \in \left(x-\epsilon, x+\epsilon\right)$, as desired.

\subsection*{Part 2: Building Sequence $\{n_k\}$ from Identified Primes}

Now, let $I$ index the primes used in all the sums of the form $\frac{\pi}{G}\sum_{j=b}^t \frac{1}{p_j^2}$ and $J$ index the primes and multiplicities used in any $\psi(p^k)$ in the special cases.  Let $n = \prod_{j\in I} p_j$ and let $m =\prod_{j\in J} p_j^{k_j}$. (If $I$ is empty, instead let $I$ index a sufficiently large prime $p$ so that $\frac{\pi}{p^2G}+S_k$ is still within $\epsilon$ of $x$, and if $J$ is empty, instead let $m=1$).  

Now, if $\lambda(m) = -1$, replace $m$ by $pm$ for some $p$ indexed by $I$ ($\dagger$).  Then for all $k\in \mathbb{N}$, let $n_k =mn^{2k}$, noting that $\lambda(n_k)=1$, and as $k\to\infty$ the contribution to $\theta(n_k)$ by the primes indexed by $I$ will converge on $\frac{\pi}{G}\sum_{j\in I}\frac{1}{p_j^2}$, and so $\{\theta(n_k)\}$ will have infinitely many terms within the interval $\left(x-\epsilon,x+\epsilon\right)$, as desired.   Hence, $\{\theta(n)\}$ is dense in $(0,\pi)$.

A similar argument, adjusting line $(\dagger)$ appropriately, shows there exists an infinite sequence of integers $\{n_k\}$, with each $n_k$ having $\lambda(n_k)=-1$, which sequence $\{\theta(n_k)\}$ converges to any $x\in(\pi, 2\pi)$.  Hence we have finally that $\{w(n)\}$, $n\in\mathbb{N}$, is dense in the unit circle in $\mathbb{C}$, as desired.
\end{proof}
\end{theorem}

\clearpage
\begin{nolinenumbers}
\section*{Appendix B: Sample Code}\label{appen_sample_code}
In this section we present sample SageMath code\footnote{Special thanks to Dr. William Stein for the excellent open-source SageMath software, which gives everyone access to powerful mathematical computing.  Nicely done, sir.} which utilizes the algorithm from the full proof of the Density Theorem to identify primes which can be used to produce an infinite sequence $\{n_k\}$ such that for sufficiently large $k$, each $\theta(n_k)$, principal argument of $w(n_k)$, is found in a specified interval.  The default behavior is to choose 100 target values randomly from $(0,\pi)$ and automatically find lists of primes and their exponents within a specified $\epsilon$ of each target value.  This sample code only uses the first 10,000 primes and so $\epsilon$ is restricted to be relatively large (0.0001), but these values can easily be modified with sufficient computer resources.

There are two cells.  The first one should be run only once, to set up data structures that will be used repeatedly.

\subsection*{Code to Run Once for Initialization}
\begin{lstlisting}
# this section sets up a list of primes to work with
# the more primes you have, the smaller the epsilon you can use
# however, it requires more memory and processing time, etc
#
# this cell should be run only once, prior to the cell below

P = Primes()
no_of_primes = 10000
first_tenk = list(P[0:no_of_primes])

# terms is a list which contains the reciprocals of 
# the primes squared. in other words, the terms of 
# the series P(2) (the prime zeta function)
terms = [1/prime**2 for prime in first_tenk]


# we manually set an approximation to the limit 
# value of P(2). In the paper we call this value G.
# note that denormalizing using this approximate value  
# means that we risk an index-out-of-value error if we 
# choose an epsilon too small (around the size of 
# the distance between this and the true partial sum 
# that we are working if we only use no_of_primes primes)
G = 0.452247420041065

# here we set up the partial sums of the terms
# we also set up the "tails", which are the series
# sum G minus everything from the indexed prime onward

# this is a cheap way of creating a list of the correct length
partials = list(terms)

tails = list()
tails.append(G) 
# this manually makes the whole series sum G the first "tail"

for i in range(1,no_of_primes):

    # add current term to previous partial sum
    partials[i]= partials[i-1]+terms[i] 
    # the current tail is G minus the previous partial sum
    tails.append(G-partials[i-1])  
\end{lstlisting}

\subsubsection*{Run Selection Rounds on Batch of Random Target values $x$}
\begin{lstlisting}
#####################################################
#
#
# This SageMath code is designed to calculate primes 
# to use for a large set of target values of x in (0,pi)
# the idea is to have a dataset to examine as necessary
# (of course, individual values can be manually entered)
#
# This code follows the "Part 1: Prime Selection Rounds" 
# algorithm from the proof of the Density Theorem in the paper
#
# first we choose an x in the interval (0,pi) and a 
# small positive epsilon.  remember that if epsilon is very 
# small it will require more primes and a correspondingly 
# more precise approximation to G in the initialization code
#
# x is referred to as the 'normalized' value because it is 
# in the range of (0,pi) as opposed to y which is 'unnormalized'
# because it is in the range of (0,G)
#
# ##########################################################
#
# for each value of x we want to record the following:
#     x, eps (x and epsilon in range (0,pi))
#     y, epg (unnormalized x, eps in range (0,P(2)) ie "G")       
#     finite set of primes to use in infinite multiplicity
#     set of primes to use in fixed multiplicity 
#         and their multiplicities
#     lowest prime for final "tail"
#         this corresponds to "Termination Case" in algorithm
#         where we land in the interval after choosing B_k
#     unnormalized final error (in range (0,G=P(2)) )
#     b and t values 


import numpy as np

# setup data structures

# how many x values to generate in the range (0,pi)
sample_size = 100

# create list of random numbers in (0,pi)
x_list = pi*np.random.random_sample(size = sample_size)

# or manually pick your favorite
# x_list = [0.542678282825539*pi]

# we multiply by pi in both cases above even though below 
# we are dividing by pi anyway, because this allows us to 
# report the actual value of x if desired

eps = 0.0001
verbose     = False
output_flag = False

# uncomment this line if you want extra info as the 
# algorithm is running during each Round: 
# verbose     = True 
# note that if you have a large batch of 
# random x's this will be an awful lot of info.  it is 
# better for a small batch, just to follow what's happening:
  
# you always get a list of indices b_k and t_k but
# uncomment this line if you want the prime list printed
# at the end of each search (with a given x):
# output_flag = True       

all_bs       = list()
all_ts       = list()
all_mults    = list()
errors       = list()

for run_number in range(len(x_list)):
    x = x_list[run_number]

    # both variables are in the range of 0 to pi
    # unnormalize both so they're in the range of (0,G)
    # where G is the sum of the reciprocals of primes squared
    y = x*(G/pi)
    epg = eps*(G/pi)

    # the upper limit of our unnormalized epsilon neighborhood is:
    nbh_top = y+epg
    # this is the right end of our target interval
    # the running sum after each round should never exceed this

    prime_flag          = False
    Bnaught             = list()
    Tnaught             = list()
    ess                 = list()
    prime_list          = list()
    fixed_primes        = list()
    mults               = list()
    fixed_primes_cases  = list()
    fixed_mults         = list()
    b_indices           = list()
    t_indices           = list()


    iter = 0
    
    # Begin Round 0

    ess.append(0)
    error = y

    ess_sum = sum(ess)

    # to begin with we want to find a sum that lands us 
    # at least above the bottom limit of the epsilon 
    # neighborhood, ie, the target interval
    cutoff = y-epg

    print('')
    print('------------------------------------------------------')
    print('Initialization complete.  Beginning search for primes.')
    print('This is run # '+str(run_number+1)+' out of '+str(sample_size)+' runs.')
    print('------------------------------------------------------')
    print('')

    if verbose:
        print('In this run we are finding primes for x = '+str(x))
        print('in the range of (0,pi)')
        print('and epsilon = '+str(eps))
        print('')    
        print('but we actually use unnormalized values')
        print('that is, values in the range of (0,G), which are:')
        print(N(y))
        print(N(epg))
        print('for unnormalized x and epsilon, respectively')
        print('')
        print('We begin with an error equal to unnormalized y, error ='+str(N(error)))
        print('')
        print('')

    while error > epg:
        # this is where we repeatedly run the selection round
        # as long as we are still under the left edge of the
        # target interval
        if verbose:
            print('')
            print('Beginning Round number '+str(iter)) 
            print('')
            print('Current S value: '+str(N(ess_sum)))
            print('')


        # find last prime such that the tail of the partials 
        # is greater than x-eps
        # we'll start by adding this entire tail to our sum S
        #
        if verbose:
            print('We want to find a tail to add to S so that S is bigger than the x - eps (unnormalized).')
            print('The current distance to the lower end of the epsilon neighborhood is '+str(N(cutoff)))
            print('')

        b_index = 0
        while tails[b_index] > cutoff:
            b_index+= 1

        b_index = b_index-1
            # the while loop stops with an index which is the
            # first tail which is too small, so we back it up
            # to get the last one that is big enough
            # in the paper this would be b_k if we are in Round k

        # save this b_index to our list
        b_indices.append(b_index)

        if verbose:
            print('the largest prime such that the tail is bigger than x-eps is the '+str(b_index+1)+'th prime')
            # we fix the index to be human readable so that
            # 2 is called the first prime rather than the 0th
            print(first_tenk[b_index])
            # I never got around to fixing it so it doesn't refer
            # to everything as the 'th' prime...oh well.

            print('and the tail in question is')
        Bnaught.append(tails[b_index])
        # append the tail to the list of B-tails
        if verbose:
            print(Bnaught[iter])
                # we save all of these for sanity-checking if needed
            print('which should be bigger than the cutoff of')
            print(N(cutoff))
                # this is just a sanity check


        # we want to subtract off the smallest tail 
        # which itself is larger than the diff between
        # our current sum S (including the lastest B-tail) 
        # and the upper limit of our epsilon-nbhd
        # this calculates how far we have exceeded the 
        # upper limit of the interval by
        cutoff_dist = ess_sum+Bnaught[iter] -nbh_top

        # we need to subtract off the smallest tail that 
        # is bigger than S - (x+eps) (but in unnormalized world)

        if verbose:
            print('')
            print('**************************************************')
            print('We will add that tail to S and then we look for a tail to cut off that is ')
            print('larger than the difference between our new S and x+eps')
            print('which is '+str(N(cutoff_dist)))

        # by our choice of the Btail we are above the lower limit
        # of the epsilon-nbhd. if S-(x+eps) is negative, it means
        # we are below the upper limit, which means we can stop
        # because we've found a list of primes (infinite at this
        # point but that doesn't matter) which places us in the nbhd
        # this is referred to in the proof as the termination case
        if cutoff_dist < 0:
            if verbose:
                print('Oh wait---with the last choice of B we are in the epsilon nbh')
                print('with a value of '+str(N(ess_sum+Bnaught[iter])))
            ess.append(Bnaught[iter])
            prime_flag = True
            primes_beyond = first_tenk[b_index]
            # it would be nice to at this point find a tail
            # which is smaller than the value S - (x-eps)
            # and to cut that tail off so we have a finite 
            # list of primes, but we don't actually bother doing
            # that here.  as long as we cut off a tiny enough bit
            # it works.  in the code, we are automatically cutting 
            # it off after the 10,000th prime anyway

            # in this case we record the t_index as Infinity
            t_indices.append(oo)
            mults.append(oo)
        else:
            # if we pass the termination case we need to look
            # for a T_k tail to truncate our B_k so it doesn't
            # exceed the upper limit of the target interval
            t_index = 0
            while tails[t_index] > cutoff_dist:
                t_index+= 1

            if verbose:
                print('')
                print('the first prime such that the tail is smaller than S - (x+eps) is the '+str(t_index+1)+'th prime')
                # the +1 is so we call 2 the 1st prime
                print(first_tenk[t_index])

                print('and the tail in question is')
            Tnaught_first_smaller = tails[t_index]
            if verbose:
                print(Tnaught_first_smaller)
                print('which should be smaller than the distance')
                print('by which we have exceeded the target nbhd of')
                print(N(cutoff_dist))
                # recall cutoff is how far too much to the right
                # we landed with our last B_k

            t_plus_one = t_index-1
            # this is "t_0 + 1" from the algorithm in the proof
            # it is the raw index of the first prime we start with to
            # form a tail that is to be removed
            # so that the last prime we include in the truncated tail
            # will be t_0 (or t_k in Round k) which we call t_act
            t_act = t_index -2

            # this tail is the smallest which is itself big enough
            # that it's bigger than 'cutoff', the distance by which
            # we overshot the interval.  it is T_k in Round k
            Tnaught.append(tails[t_plus_one])
            
            if verbose:
                print('and the overshoot distance above should be')
                print('smaller than the tail we plan to remove:')
                print(N(Tnaught[iter]))
                print('which is the tail starting from the prime '+str(first_tenk[t_plus_one]))
                print('so we will remove that and larger primes')

            # we append the actual t index to our list
            # t_act is called t_0 in the algorithm in the proof
            if t_plus_one > b_index: # main case
                t_indices.append(t_act)
                index_incl = t_act
            elif t_plus_one == b_index: # special case
                t_indices.append(t_plus_one)
                index_incl = t_plus_one
                # if it's a special case we need to fix the index

            if verbose:
                print('meanwhile appending raw t-index:'+str(index_incl))
                print('this is the raw index of the last prime')
                print('being included in this list')
                print('meanwhile the b-index is '+str(b_index))
                print('current list of b-indices is ')
                print(b_indices)
                print('and current list of t-indices is ')
                print(t_indices)
                
            # if b_index = actual t-index, we are in what is
            # called the 'special case' in the paper, which
            # means that we are only looking at
            # a single prime (which should be 2 or 3)
            # if that's the case, we need to set its multiplicity
            # but in the main case the multiplicity is set to oo
            if Bnaught[iter]-Tnaught[iter] > 0:
                # if tail to add != tail to subtract
                # this is the 'main case'
                ess.append( Bnaught[iter] - Tnaught[iter])
                prime_list.extend(first_tenk[b_index:t_plus_one])
                # the prime indexed by 't_plus_one' is NOT included
                # in this slice---it is the first prime
                # in the tail that is being removed

                # this is the case where all the primes used
                # have infinite multiplicity (for now)
                mults.append(oo)
            else:
                this_prime = first_tenk[b_index]
                # this is the 'special case' where 
                # t-plus-one = b-index
                # and indexes a prime which needs its 
                # multiplicity fixed

                if verbose:
                    print('')
                    print('The smallest tail that''s big enough is the  same as the smallest that''s big enough to get cut off')
                    print('that is, our b and t indices are the same')
                    print('In this special case we would have to cut off the whole tail.  Instead we will fix multiplicity of this prime, '+str(this_prime))
                ###############################################################
                # this code handles the 'special case'
                # it should only ever run on prime 2 or 3
                #
                # we need to add, instead of a tail, 
                # (1/p^2)(1-((p-1)/p)^k) for some value k
                k = 1
                if verbose:
                    print('distance to nbd is '+str(N(cutoff)))
                special_case = (1/this_prime^2)*(1-( (this_prime - 1)/this_prime)^k)
                if verbose:
                    print('smallest available term to add is '+str(N(special_case)))
                    
                while (1/this_prime^2)*(1-( (this_prime - 1)/this_prime)^k)<cutoff:
                    k+=1
                k = k-1
                if verbose:
                    print('We will use k = '+str(k))
                special_case = (1/this_prime^2)*(1-( (this_prime - 1)/this_prime)^k)
                if verbose:
                    print('so we add on '+str(N(special_case)))
                ess.append(special_case)
                fixed_primes.append(this_prime)
                # record the special case primes which are used
                # to form 'm' in part 2
                fixed_primes_cases.append(special_case)
                # also record its multiplicity
                mults.append(k)
                fixed_mults.append(k)
                if verbose:
                    print('appending '+str(k)+' to list of mults')

        if verbose:
            print('if things have gone well we should be closer to or in the epsilon nbh')
            print('with our sum S += B-T or S += psi(new_prime^k)')

        # update running sum
        ess_sum = sum(ess)
        if verbose:
            print('S = '+str(N(ess_sum)))
            print('')
        
        # set the new distance to the target x (actually y)
        # in the paper this is called epsilon_k in Round k
        error=y-ess_sum
        # set also the distance needed to get to the
        # target interval
        cutoff = y-epg-ess_sum
        if verbose:
            print('')
            print('The unnormalized epsilon is '+str(N(epg)))
            print('and the current distance to the unnormalized x is '+str(N(error)))
        if error>epg:
            if verbose:
                print('This is outside the epsilon radius---repeat entire process...')
                print('')
                print('Repeating...')
                print('oooooooooooooooooooooooooooooooooooooo')
            iter+=1
        else:
            if output_flag:
                print('')
                print('')
                if len(prime_list)>= 1:
                    print('We will use the following primes in infinite multiplicity:')
                    print(prime_list)
                if len(fixed_primes) >= 1:
                    print('We will use the following primes with set multiplicites:')
                    for i in range(len(fixed_primes)):
                        print(str(fixed_primes[i])+' with multiplicity '+str(fixed_mults[i]))
                if prime_flag:
                    print('And all the primes past '+str(primes_beyond))
                    print('(with some sufficiently small tail cut off)')

    sumtotal = 0
    for p in prime_list:
        sumtotal += 1/p^2
    for p in fixed_primes_cases:
        sumtotal += p
    if prime_flag:
        sumtotal += tails[b_index]
    renorm = sumtotal*pi/G
    diff_from_x = N(abs(renorm-x))
    errors.append(diff_from_x)

    if verbose:
        print('')
        print('')
        print('###################################################')
        print('')
        print('Time for a sanity check:')
        print('our final value was '+str(N(sumtotal)))
        print('with a normalized value of '+str(N(renorm)))
        print('its distance from x is '+str(diff_from_x))
        print('epsilon was '+str(eps))
        print('')
        print('###################################################')
        print('###################################################')

    # add identified primes and multiplicities (which may be
    # infinity at this point) to our running list
    all_bs.append(b_indices)
    all_ts.append(t_indices)
    all_mults.append(mults)
    
    # since we are definitely below the upper limit of the 
    # epsilon nbhd, if the error exceeds epsilon then we are
    # still below the nbhd and the while loop brings us back 
    # for another round
    


###################################################
# summary of output for all runs

from tabulate import tabulate

for i in range(len(x_list)):
    table = list()
    print('')

    print('********************************************')
    print('sample value # '+str(i+1))
    print('')
    print('x-value:                  '+str(x_list[i]))
    print('had an absolute error of: '+str(errors[i]))
    print('using the primes with indices below')
    for j in range(len(all_bs[i])):
        table.append([all_bs[i][j]+1, all_ts[i][j]+1, all_mults[i][j]])
        # the +1 is to make the iiceses human-readable
    print(tabulate(table, headers=["b_indices", "t_indices", "multiplicities"]))
\end{lstlisting}
\end{nolinenumbers}

\end{document}